\documentclass[12pt,reqno]{amsart}

\usepackage{graphicx,amsfonts,amssymb,amsmath,amsthm,xpatch}
\usepackage{fullpage}
\usepackage[svgnames]{xcolor}
\usepackage{hyperref}
\hypersetup{
  colorlinks = true,
  linkcolor =red,
  anchorcolor = red,
  citecolor = DarkGreen,
  urlcolor = blue
}
\usepackage[all,arc]{xy}
\usepackage{tikz-cd,float}
\usepackage{mathrsfs,mathtools}
\usepackage{caption}
\usepackage{enumitem}
\usepackage{pdflscape,bm}
\usepackage{subfig}

\newtheorem{theo}{{\sc Theorem}}[section]

\newtheorem{defn}[theo]{{\sc Definition}}

\newtheorem{cor}[theo]{{\sc Corollary}}

\newtheorem{lem}[theo]{{\sc Lemma}}
\newtheorem{prop}[theo]{{\sc Proposition}}

\newenvironment{rem}{\medskip\noindent{\it Remark:\/} }{\medskip}

\allowdisplaybreaks

\makeatletter
\def\blfootnote{\gdef\@thefnmark{}\@footnotetext}
\def\@cite#1#2{[\textbf{#1}\if@tempswa, #2\fi]}
\newcounter{proofpart}
\xpretocmd{\proof}{\setcounter{proofpart}{0}}{}{}
\newcommand{\proofpart}[2]{%
  \par
  \addvspace{\medskipamount}%
  \stepcounter{proofpart}%
  \noindent\emph{#1 \theproofpart: #2}\par\nobreak\smallskip
  \@afterheading
}
\makeatother

\DeclareFontFamily{U}{mathx}{\hyphenchar\font45}
\DeclareFontShape{U}{mathx}{m}{n}{
      <5> <6> <7> <8> <9> <10>
      <10.95> <12> <14.4> <17.28> <20.74> <24.88>
      mathx10
      }{}
\DeclareSymbolFont{mathx}{U}{mathx}{m}{n}
\DeclareFontSubstitution{U}{mathx}{m}{n}
\DeclareMathAccent{\widecheck}{0}{mathx}{"71}
\DeclareMathAccent{\wideparen}{0}{mathx}{"75}

\title{On the Intertwining Map between Coulomb and Hyperbolic Scattering}

\author{Nicholas Lohr}
\address{Northwestern University, Evanston, IL, 60208}
\email{nlohr@math.northwestern.edu}
\date{\today}

\begin{document}

\blfootnote{The author was partially supported by NSF RTG grant
  DMS-2136217.\par
  The author has no competing interests to declare that are relevant to the content of this article.}
\begin{abstract}
We construct a unitary operator between Hilbert spaces of generalized eigenfunctions of Coulomb operators and the Laplace-Beltrami operator of hyperbolic space that intertwines their respective Poisson operators on $L^2(\mathbb{S}^{d-1})$. The constructed operator generalizes Fock's unitary transformation, originally defined between the discrete spectra of the attractive Coulomb operator and the Laplace-Beltrami operator on the sphere, to the setting of continuous spectra. Among other connections, this map explains why the scattering matrices are the same in these two different settings, and it also provides an explicit formula for the Poisson operator of the Coulomb Hamiltonian.
\end{abstract}
\vskip -4\baselineskip
\maketitle
\section{Introduction}
It is well-known that the scattering matrix $S_{\mathbb{H}^d}(\lambda):L^2(\mathbb{S}^{d-1}) \to L^2(\mathbb{S}^{d-1})$ for the nonnegative Laplace-Beltrami operator $-\Delta_{\mathbb{H}^d}$ of 
hyperbolic space, $\mathbb{H}^d, d \geq 2$, up to constant phase factors, is a pseudodifferential operator with Schwartz
kernel and functional calculus expression
\begin{equation}\label{eq:Smatrix}
S_{\mathbb{H}^d}(\lambda; \theta,\theta')=2^{-2\lambda i}\pi^{-\frac{d-1}{2}}\frac{\Gamma(\frac{d-1}{2}-\lambda i)}{\Gamma(\lambda i)}\frac{1}{|\theta-\theta'|^{d-1-2\lambda i}},\qquad S_{\mathbb{H}^d}(\lambda)=\frac{\Gamma(A+\frac{d-1}{2}-\lambda i)}{\Gamma(A+\frac{d-1}{2}+\lambda i)},
\end{equation}
respectively, where $A \coloneqq
\sqrt{-\Delta_{\mathbb{S}^{d-1}}+\frac{(d-2)^2}{4}}-\frac{d-2}{2} \geq 0$ and
$\lambda\geq 0$ parametrizes the spectrum of
$-\Delta_{\mathbb{H}^d}$ by
$\operatorname{spec}(-\Delta_{\mathbb{H}^d})=\{\lambda^2+\frac{(d-1)^2}{4}\mid\lambda
\geq 0\}$ (cf. \cite[\S 4]{B10} where
$s=\frac{d-1}{2}+\lambda i$ in his notation). It is less known that the
semiclassical scattering matrix
$S_{\widehat{H}_{\hbar}}(\hbar,\lambda)$ for the (attractive) Coulomb operator (also known as the spinless, nonrelativistic and unit-normalized hydrogen atom)
\begin{equation}\label{eq:operator}
\widehat{H}_{\hbar}\coloneqq -\frac{\hbar^2}{2}\Delta_{\mathbb{R}^d}-\frac{1}{|x|}
\end{equation}
is the \textit{same} as
$S_{\mathbb{H}^d}(\lambda)$ (see \cite[\S 4.2]{Y97} where $E=\frac{1}{2\hbar^2 \lambda^2},\alpha=-\lambda,k=\frac{1}{\hbar^2\lambda},\gamma=-\frac{2}{\hbar^2},\delta=\frac{d-1}{2}$ in his notation). Namely, for any $\hbar,\lambda>0$,
\begin{equation}\label{eq:samee}
S_{\widehat{H}_{\hbar}}(\hbar,\lambda)=S_{\mathbb{H}^d}(\lambda),
\end{equation}
where the absolutely continuous
spectrum of $\widehat{H}_{\hbar}$ is parametrized by
$$\operatorname{spec}_{\text{ac}}(\widehat{H}_{\hbar})=\{0\} \cup
\Big\{E_{\lambda}(\hbar) \coloneqq \frac{1}{2\hbar^2\lambda^2}\mid \lambda>0\Big\}.$$ Furthermore, the so-called differential of the quantum cross sectional area
$\sigma_d(\lambda)d\lambda$ for the Coulomb operator is, up to a $d$-dimensional constant, the
same as the spectral measure
$|\mathfrak{c}(\lambda)|^{-2}d\lambda$ of $\mathbb{H}^d$ where
$\mathfrak{c}(\lambda)=\frac{2^{d-2}\Gamma(\frac{d}{2})\Gamma(\lambda i)}{\sqrt{\pi}\Gamma(\frac{d-1}{2}+\lambda i )}$ is the Harish-Chandra function for $\mathbb{H}^d$
(see \cite[(1.2)]{Y97} and \cite[(3.1.3)]{B92}, respectively).
\par In this article, we explain this connection by constructing an explicit unitary operator
$\mathcal{V}_{\hbar,\lambda}$ between Hilbert spaces of
generalized eigenfunctions of $\widehat{H}_{\hbar}$ and
$-\Delta_{\mathbb{H}^d}$ with $L^2$ boundary data at the sphere at `infinity' that conjugates the respective Poisson
operators. Namely, for $\hbar>0,\lambda \in \mathbb{R}_{\neq
  0}$, and $d \geq 2$, define the Hilbert spaces
$\mathscr{H}_{\widehat{H}_{\hbar}}(\hbar,\lambda)$ and
$\mathscr{H}_{\mathbb{H}^d}(\lambda)$ of generalized
eigenfunctions in the variables $x\in \mathbb{R}^d$ and $u \in
\mathbb{H}^d$ by
\begin{equation}\label{eq:Hilbert}
\mathclap{
\begin{aligned}
\mathscr{H}_{\widehat{H}_{\hbar}}(\hbar,\lambda) &\coloneqq 
                                                   \Big\{ F_{\widehat{H}_{\hbar}}^{\hbar,\lambda}(x)
  \coloneqq
                                                   \int_{\mathbb{S}^{d-1}}\psi_{\hbar,\lambda}(x;\theta)f(\theta)d\theta \mid f
                                                   \in
                                                   L^2(\mathbb{S}^{d-1})
                                                   \Big\},\ \lVert F_{\widehat{H}_{\hbar}}^{\hbar,\lambda}\rVert_{\mathscr{H}_{\widehat{H}_{\hbar}}(\hbar,\lambda)}\coloneqq \lVert f \rVert_{L^2},\\
\mathscr{H}_{\mathbb{H}^d}(\lambda) &\coloneqq \Big\{F_{\mathbb{H}^d}^{\lambda}(u) \coloneqq  \int_{\mathbb{S}^{d-1}}e_{\lambda}(u;\theta)f(\theta)d\theta\mid f \in L^2(\mathbb{S}^{d-1}) \Big\},\ \lVert F_{\mathbb{H}^d}^{\lambda}\rVert_{\mathscr{H}_{\mathbb{H}^d}(\lambda)}\coloneqq \lVert f \rVert_{L^2},
\end{aligned}}
\end{equation}
where $\psi_{\hbar,\lambda}(\bullet; \theta)$ and
$e_{\lambda}(\bullet;\theta)$ are the plane waves for the Coulomb operator
and hyperbolic space, respectively (defined in
\eqref{eq:ppw2},\eqref{eq:hypplane2}). For $f \in
L^2(\mathbb{S}^{d-1})$, we have
\begin{equation}\label{eq:forjaredsunderstanding}
\widehat{H}_{\hbar}F_{\widehat{H}_{\hbar}}^{\hbar,\lambda}=E_{\lambda}(\hbar)F_{\widehat{H}_{\hbar}}^{\hbar,\lambda},
\qquad
-\Delta_{\mathbb{H}^d}F_{\mathbb{H}^d}^{\lambda}=\Big(\lambda^2+\frac{(d-1)^2}{4}
\Big)F_{\mathbb{H}^d}^{\lambda}, 
\end{equation}
and hence $F_{\widehat{H}_{\hbar}}^{\hbar,\lambda}$ and
$F_{\mathbb{H}^d}^{\lambda}$ lie in the absolutely continuous
spectrum of $\widehat{H}_{\hbar}$ and $-\Delta_{\mathbb{H}^d}$,
respectively. Let
\begin{equation}\label{eq:poissonops}
\begin{aligned}
P_{\widehat{H}_{\hbar}}(\hbar,\lambda)&:L^2(\mathbb{S}^{d-1})
  \to \mathscr{H}_{\widehat{H}_{\hbar}}(\hbar,\lambda),\quad  f
  \mapsto F_{\widehat{H}_{\hbar}}^{\hbar,\lambda}\\P_{\mathbb{H}^d}(\lambda)&:L^2(\mathbb{S}^{d-1})
  \to\mathscr{H}_{\mathbb{H}^d}(\lambda),\quad f \mapsto
  F_{\mathbb{H}^d}^{\lambda}
  \end{aligned}
\end{equation}
denote the Poisson operators (see \cite[\S 4]{B10} where
$P_{\mathbb{H}^d}(\lambda)=E_0(s)$ in his notation). With the
operators
\begin{align}
  \mathcal{I}[f] &\coloneqq f\Big(\frac{\bullet}{|\bullet|^2}\Big),\qquad  M[f] \coloneqq
                         \bigg|\frac{|\bullet|-1}{2}\bigg|^{\frac{d+1}{2}}f,\qquad  R_{\{|\bullet|>1\}} [f] \coloneqq f|_{\{|\bullet|>1\}},\label{eq:otherops}\\
  \widehat{\mathcal{D}}_{\frac{1}{\hbar\lambda}}[f]&\coloneqq
                                                     \frac{1}{|\hbar\lambda|^{\frac{d}{2}}}f\Big(\frac{\bullet}{\hbar\lambda}\Big),\qquad \mathcal{F}_{\hbar}[f] \coloneqq \frac{1}{(2\pi \hbar)^{\frac{d}{2}}}\int_{\mathbb{R}^d}f(x)e^{-\frac{i}{\hbar}\langle x,\bullet \rangle}dx \label{eq:mainops},
\end{align}
we prove the main theorem of this article:
\begin{theo}\label{th:main} For every $\hbar>0$ and $\lambda \in \mathbb{R}_{\neq 0}$, the unitary
operator
$\mathcal{V}_{\hbar,\lambda}:\mathscr{H}_{\widehat{H}_{\hbar}}(\hbar,\lambda)\to
\mathscr{H}_{\mathbb{H}^d}(\lambda)$
given by
\begin{equation}\label{eq:mainmap}
\mathcal{V}_{\hbar,\lambda} \coloneqq \mathcal{I} \circ M \circ R_{\{|\bullet|>1\}}  \circ \widehat{\mathcal{D}}_{\frac{1}{\hbar\lambda}} \circ \mathcal{F}_{\hbar}
\end{equation}
makes the following diagram commute:
\begin{equation}\label{eq:diagram}\begin{tikzcd}[column sep=0.25em]
	& {L^2(\mathbb{S}^{d-1})} \\
	{\mathscr{H}_{\widehat{H}_{\hbar}}(\hbar,\lambda)} && {\mathscr{H}_{\mathbb{H}^d}(\lambda)}
	\arrow["{P_{\widehat{H}_{\hbar}}(\hbar,\lambda)}"', from=1-2, to=2-1]
	\arrow["{P_{\mathbb{H}^d}(\lambda)}", from=1-2, to=2-3]
	\arrow["{\mathcal{V}_{\hbar,\lambda}}"', from=2-1, to=2-3]
  \end{tikzcd}.\end{equation}
  \end{theo}
  \begin{rem}
It is well-known that
$\mathscr{H}_{\mathbb{H}^d}(\lambda)$ is dense in
\begin{equation*}
\mathscr{E}_{\mathbb{H}^d}(\lambda) \coloneqq \bigg\{f \in C^{\infty}(\mathbb{H}^d): -\Delta_{\mathbb{H}^d}f=\Big(\lambda^2+\Big(\frac{d-1}{2}\Big)^2\Big)f\bigg\},
\end{equation*}
where the space $\mathscr{E}_{\mathbb{H}^d}(\lambda)$ is
topologized by the seminorms $\lVert f \rVert_{m}^{K} \coloneqq
\sum_{|\boldsymbol{\alpha}|\leq m}\sup_{x \in K}|D^{\boldsymbol{\alpha}}f(x)|$ for $m \in
\mathbb{Z}_{\geq 0}$ and $K \subset \mathbb{H}^d$ compact (see \cite[Introduction, Lemma 4.20]{H84} for $d=2$ and
\cite[Chapter IV, Lemma 2.2]{H70} for the general case). Additionally, one can show with the Helgason-Fourier transform and Plancherel's theorem thereof
that $$L^2(\mathbb{H}^d) \cong
\int_{[0,\infty)}^{\oplus}\mathscr{H}_{\mathbb{H}^d}(\lambda)\frac{2^{d-1}d\lambda}{2\pi |\mathbb{S}^{d-1}||\mathfrak{c}(\lambda)|^2}. $$
With this perspective, Theorem \ref{th:main} in some sense `densely' relates
the nonzero absolutely continuous spectrum of the operator
$-\Delta_{\mathbb{H}^d}$ to that of $\widehat{H}_{\hbar}$.
\end{rem}
\par This theorem has several immediate corollaries.
\begin{cor}\label{cor:intertwine} $\displaystyle \mathcal{V}_{\hbar,\lambda}\widehat{H}_{\hbar}\mathcal{V}_{\hbar,\lambda}^{-1}=\frac{1}{2\hbar^2}\Big(-\Delta_{\mathbb{H}^d}-\frac{(d-1)^2}{4}\Big)^{-1} \text{ on }\mathscr{H}_{\mathbb{H}^d}(\lambda) \text{ for }\hbar>0,\lambda \in \mathbb{R}_{\neq 0}$.
\end{cor}
This follows since both sides act as multiplication by
$E_{\lambda}(\hbar) \coloneqq \frac{1}{2\hbar^2\lambda^2}$ on
$\mathscr{H}_{\mathbb{H}^d}(\lambda)$, by \eqref{eq:forjaredsunderstanding}. If we define the
scattering matrices by
\begin{equation}\label{eq:scatmatrices}
S_{\mathbb{H}^d}(\lambda) \coloneqq \overline{P_{\mathbb{H}^d}(-\lambda)^*}P_{\mathbb{H}^d}(\lambda),\qquad S_{\widehat{H}_{\hbar}}(\hbar,\lambda) \coloneqq \overline{P_{\widehat{H}_{\hbar}}(\hbar,-\lambda)^*}P_{\widehat{H}_{\hbar}}(\hbar,\lambda),
\end{equation}
where the bar denotes conjugating the operator by complex conjugation. We have the following corollary:
\begin{cor}\label{cor:scat} $\displaystyle S_{\widehat{H}_{\hbar}}(\hbar,\lambda)=S_{\mathbb{H}^d}(\lambda)$ for $\hbar>0,\lambda \in \mathbb{R}_{\neq 0}$.
\end{cor}
This follows from the unitarity of $\mathcal{V}_{\hbar,\lambda}$ and the fact that $\mathcal{V}_{\hbar,\lambda}=\overline{\mathcal{V}_{\hbar,-\lambda}}$. \par The next corollary is a ``hidden'' inversion symmetry in Fourier space of the Coulomb operator generalized eigenfunctions, a quantum analogue of the conservation of the Runge-Lenz vector $R^+$ in classical mechanics (see \eqref{eq:ihatejared}).
\begin{cor}\label{cor:sym} For $\psi \in \mathscr{H}_{\widehat{H}_{\hbar}}(\hbar,\lambda)$,
\begin{equation}\label{eq:invsymm}
(\widehat{\mathcal{D}}_{\frac{1}{\hbar\lambda}}\circ \mathcal{F}_{\hbar})[\psi](\xi)=-\frac{e^{-\pi |\lambda|}}{|\xi|^{d+1}}(\widehat{\mathcal{D}}_{\frac{1}{\hbar\lambda}}\circ \mathcal{F}_{\hbar})[\psi]\Big(\frac{\xi}{|\xi|^2}\Big)\quad \text{when} \quad 0<|\xi|<1.
\end{equation}
\end{cor}
This follows directly from Lemma \ref{lem:jared}. In lieu of Theorem \ref{th:main}, the appearance of $-e^{-\pi |\lambda|}$ is understood as coming from a branch cut of the hyperbolic plane waves \eqref{eq:hypplane2} when we extend these functions to the exterior of the unit ball. \par The next corollary is a Schwinger-type formula for the Poisson operator of $\widehat{H}_{\hbar}$ (in the same spirit as \cite{S64} for the Green function of $\widehat{H}_{\hbar}$ on the negative, discrete spectrum).
\begin{cor}\label{cor:formula} For $f \in L^2(\mathbb{S}^{d-1})$, 
\begin{equation*}
P_{\widehat{H}_{\hbar}}(\hbar,\lambda)[f]=\mathcal{F}_{\hbar}^{-1}\circ \widehat{\mathcal{D}}_{\hbar \lambda}\int_{\mathbb{S}^{d-1}}\frac{2^{\frac{d+1}{2}}}{|\bullet-\theta|^{d-1-2\lambda i}(|\bullet|^2-1)^{^{1+\lambda i}}} f(\theta)d\theta,
\end{equation*}
where the inverse semiclassical Fourier transform is suitably interpreted (see Lemma \ref{lem:jared}).
\end{cor}
Again, this follows directly from Theorem \ref{th:main}.
\par The next corollary defines the Coulomb `regular' representation and proves it is irreducible, explaining the $\operatorname{SO}_{e}(1,d)$-symmetry of the positive energy spectrum of the Coulomb operator.
\begin{cor}\label{cor:group}
The infinite dimensional unitary representation
$\rho_{\widehat{H}_{\hbar}}(\hbar,\lambda):\operatorname{SO}_{e}(1,d)
\to \operatorname{U}(\mathscr{H}_{\widehat{H}_{\hbar}}(\hbar,\lambda))$ of the restricted Lorentz group $\operatorname{SO}_{e}(1,d)$ given
by $$\rho_{\widehat{H}_{\hbar}}(\hbar,\lambda) \coloneqq
\mathcal{V}_{\hbar,\lambda}^{-1}\circ
\rho_{\mathbb{H}^d}(\lambda) \circ \mathcal{V}_{\hbar,\lambda},$$
where $\rho_{\mathbb{H}^d}(\lambda)[A](f)\coloneqq f(A^{-1}\bullet)$, is irreducible.
\end{cor}
This follows directly from Theorem \ref{th:main} and the fact that $\rho_{\mathbb{H}^d}(\lambda)$ is irreducible, proved in
\cite[Chapter III, Proposition 5.3]{H70}.
\subsection{Repulsive Coulomb Potential}
The same theorem and corollaries can be shown for the repulsive
Coulomb operator
\begin{equation}\label{eq:HamilRep}
\widehat{H}_{\hbar}^{-} \coloneqq
-\frac{\hbar^2}{2}\Delta+\frac{1}{|x|}
\end{equation}
by constructing similar
perturbed plane waves to $\psi_{\hbar,\lambda}(\bullet ; \bullet)$ (see \eqref{eq:future} for more details). If we define the Hilbert
space $\mathscr{H}_{\widehat{H}_{\hbar}^{-}}(\hbar,\lambda)$ and the Poisson operator
$P_{\widehat{H}_{\hbar}^{-}}(\hbar,\lambda)$ analogously to
\eqref{eq:Hilbert} and \eqref{eq:poissonops}, we have the
following theorem:
\begin{theo}\label{th:main2}The unitary
operator
$\mathcal{V}_{\hbar,\lambda}^{-}:\mathscr{H}_{\widehat{H}_{\hbar}^{-}}(\hbar,\lambda)\to
\mathscr{H}_{\mathbb{H}^d}(\lambda)$
given by
\begin{equation*}
\mathcal{V}_{\hbar,\lambda}^{-} \coloneqq M \circ R_{\{|\bullet|<1\}} \circ \widehat{\mathcal{D}}_{\frac{1}{\hbar\lambda}} \circ \mathcal{F}_{\hbar}
\end{equation*}
makes the following diagram commute:
\[\begin{tikzcd}[column sep=0.25em]
	& {L^2(\mathbb{S}^{d-1})} \\
	{\mathscr{H}_{\widehat{H}_{\hbar}^{-}}(\hbar,\lambda)} && {\mathscr{H}_{\mathbb{H}^d}(\lambda)}
	\arrow["{P_{\widehat{H}_{\hbar}^{-}}(\hbar,\lambda)}"', from=1-2, to=2-1]
	\arrow["{P_{\mathbb{H}^d}(\lambda)}", from=1-2, to=2-3]
	\arrow["{\mathcal{V}_{\hbar,\lambda}^{-}}"', from=2-1, to=2-3]
  \end{tikzcd}.\]
  \end{theo}
The proof of Theorem \ref{th:main} can be easily adapted to
prove this theorem, and we leave the details for the reader. With Theorem \ref{th:main2}, one can show
the analogous corollaries from the previous
section. With Theorems \ref{th:main} and \ref{th:main2}, we have the following corollary:
\begin{cor}
If we define the unitary operator $$\mathcal{T}_{\hbar,\lambda} \coloneqq
(\mathcal{V}_{\hbar,\lambda}^{-})^{*}\circ
\mathcal{V}_{\hbar,\lambda}=\mathcal{F}_{\hbar}^{-1}\circ
\widehat{\mathcal{D}}_{\hbar\lambda}\circ \mathsf{M} \circ \mathcal{I} \circ
\widehat{\mathcal{D}}_{\frac{1}{\hbar\lambda}}\circ \mathcal{F}_{\hbar},$$
where $\mathsf{M}[f] \coloneqq \frac{1}{|\bullet|^{d+1}}f$, then we
have the following commutative diagram:
\[\begin{tikzcd}[column sep=0.25em]
	& {L^2(\mathbb{S}^{d-1})} \\
	{\mathscr{H}_{\widehat{H}_{\hbar}}(\hbar,\lambda)} && 	{\mathscr{H}_{\widehat{H}_{\hbar}^{-}}(\hbar,\lambda)} 
	\arrow["{P_{\widehat{H}_{\hbar}}(\hbar,\lambda)}"', from=1-2, to=2-1]
	\arrow["{P_{\widehat{H}_{\hbar}^{-}}(\hbar,\lambda)}", from=1-2, to=2-3]
	\arrow["{\mathcal{T}_{\hbar,\lambda}}"', from=2-1, to=2-3]
  \end{tikzcd}.\]
\end{cor}
This corollary reflects the inversion phenomenon in the
classical setting described in \S \ref{sec:moser}.
\subsection{Prior Work}\label{sec:prior}
\par The attractive Coulomb operator corresponds to the
classical phase space Hamiltonian
\begin{equation*}
H^+:T^*(\mathbb{R}^d\setminus\{0\}) \to \mathbb{R},\quad H^+(x,\xi)\coloneqq \frac{|\xi|^2}{2}-\frac{1}{|x|}
\end{equation*}
called the Kepler Hamiltonian, where we identify
$T^*(\mathbb{R}^d\setminus\{0\})=T(\mathbb{R}^d\setminus\{0\})=\mathbb{R}^d\setminus\{0\}
\times \mathbb{R}^d$ using the Riemannian metric on
$\mathbb{R}^d\setminus\{0\}$. For a fixed energy $E$, the
Hamiltonian orbits, also called Kepler orbits, lie on the energy
hypersurface
\begin{equation*}
\Sigma_E^+ \coloneqq \{(x,\xi) \in T^*(\mathbb{R}^d \setminus \{0\}) \mid H^+(x,\xi)=E\}.
\end{equation*}
For any energy $E \in \mathbb{R}$, $\Sigma_E^+$ is \textit{not} compact due to the $|x| \to 0, |\xi| \to \infty$ regime.
For $E<0$, the orbits consist of two types: bounded periodic orbits
whose configuration space projections are planar ellipses, and
unbounded ``collision'' orbits whose configuration space projections
are line segments terminating at the origin in finite
time. The configuration space projections of the periodic
Kepler orbits follow Kepler's laws of planetary motion (with one
body fixed and all physical constants fixed to $1$). Namely, the periodic configuration space trajectories
\begin{itemize}
\item are planar ellipses with the origin fixed at one focus (in the Keplerian view, the origin represents the Sun),
\item are such that the line segment connecting the trajectory
to the origin sweeps out equal areas during equal time
intervals,
\item have period $T$ related to the energy $E$ by the
formula
\begin{equation}\label{eq:kepler3}
T=\frac{2\pi}{p_0^3},\quad p_0 \coloneqq \sqrt{-2E},
\end{equation}
\end{itemize}
where we have used our convention on physical constants. Observe that Kepler's third
law is popularly stated with the length of the semi-major axis
$a$, but with our conventions $a=p_0^{-2}$ (see
\cite[(5)]{M83} and the very nice expository article \cite{vHH09}). \par This Hamiltonian system is not only completely
integrable, but it is maximally superintegrable with $2d-1$
independent integrals of motion coming from the components of the conserved quantities of the Hamiltonian $H^+$, the angular momentum
$(d-2)$-vector ${\bf L}$, and the
Runge-Lenz eccentricity vector $R^+$ defined by
\begin{equation}\label{eq:ihatejared}
{\bf L}(x,\xi) \coloneqq \star(x \wedge \xi),\qquad R^+(x,\xi) \coloneqq \Big( |\xi|^2-\frac{1}{|x|}\Big)x-(x\cdot \xi)\xi,
\end{equation}
where $\star$ denotes the Hodge star operator. The magnitudes of
these quantities are related by the formula
$$|R^+|^2=1+2E|{\bf L}|^2.$$
A Kepler orbit is a collision orbit if and only if ${\bf L}=0$. Provided that ${\bf L } \neq 0$, in configuration space, ${\bf L}$ determines the plane of motion, $|R^+|$ is the eccentricity of the conic section, the direction of $R^+$ is, up to a $90^{\circ}$-rotation, the direction of the other focus, and $|2E|^{-2}$ is the length of the semi-major axis (as noted previously). The Runge-Lenz vector $R^+$ has a long, complicated history of discovery and rediscovery (see the works of Goldstein \cite{G75,G76}), but, most noteworthy, Hamilton in \cite{H47} showed that the Runge-Lenz vector can be understood as coming from the geometry of the momentum space projections of the Kepler orbits, which miraculously happen to be \textit{circles}, each with center $R^+/|{\bf L}|$ and radius $1/|{\bf L}|$ (more carefully, these circles degenerate into lines for the collision orbits). The superintegrability explains why the bounded orbits are not merely quasi-periodic and confined to invariant tori as guaranteed from the Liouville-Arnold theorem (see \cite[Chapter 10]{A89}), but the bound orbits are genuinely
\textit{periodic} (see \cite{GS90} for more on the symmetries of this problem). \par Because of the collision orbits, the
Hamiltonian flow of $H^+$ is \textit{not} complete. In
\cite{M70}, Moser compactified $\Sigma_E^+$ for $E<0$ and
regularized this flow with a reflection condition. That is, the
collision orbits are reflected back on the same line with the
same incoming speed when they hit the origin, resembling a
degenerate ellipse. This completes the Hamiltonian flow for
$E<0$ and extends the collision orbits past their finite
collision time to all time. He accomplished this by constructing a
energy-dependent symplectomorphism from this compactified energy hypersurface to $T_1^*\mathbb{S}^d$ that maps the regularized
Hamiltonian flow of $H^+$ to the cogeodesic flow of
$T_1^*\mathbb{S}^d$, up to a reparametrization of
time. Furthermore, the map pulls back the $\binom{d+1}{2}$
components of the angular momentum on $\mathbb{S}^d$ to the
$\binom{d}{2}$ components of {\bf L} on $\mathbb{R}^d$
\textit{along with} the $d$ (scaled) components of $R^+$ (see
\cite[Theorem 4.1]{HdL12}).
\par The `quantization' of Moser's map was defined 35 years
\textit{earlier} in \cite{F35}. That is, Fock constructed a
$\hbar,N$-dependent unitary map that sends the discrete spectrum
eigenfunctions of $\widehat{H}_{\hbar}$ with eigenvalue
$E=-\frac{1}{2\hbar^2(N+1)^2},N=0,1,\ldots$ to the
eigenfunctions of $-\Delta_{\mathbb{S}^d}$ with eigenvalue
$N(N+d-1)$. This map explains why the dimensions of these
eigenspaces coincide and furthermore completely relates the
discrete spectra of these operators. Recently in \cite{L23}, this map was used to completely characterize the set of semiclassical measures corresponding to sequences of eigenfunctions of $\widehat{H}_{\hbar}$.
\par For the $E>0$ regime, the classical and quantum mechanics were treated by later authors. Kepler's first two laws continue to
hold where the planar ellipses are replaced by planar
hyperbolas, and Belbruno and Osipov showed Moser's reflection
condition continues the regularize the positive energy
Hamiltonian flow (see \cite{B77,O77} and further references in
\cite{M83}). The codomain of Moser's map is replaced with the unit cotangent bundle of two copies of hyperbolic space, and we give a presentation in this work in \S
\ref{sec:moser}.
\par The quantum operator relating the absolutely continuous spectra of
$-\Delta_{\mathbb{H}^d}$ and $\widehat{H}_{\hbar}$ was sketched in \cite{F35}, and the forward map was more fleshed out in the
physics papers \cite{BI66,PP66}. Unfortunately, these papers do not specify the Hilbert spaces involved in the domain of the operator, and they also don't
adequately address the inverse map, which requires careful
consideration of intersecting singularities (see Lemma \ref{lem:jared} and the following remark). These papers also use deep and
nontrivial results of infinite-dimensional unitary
representation theory of $\operatorname{SO}_e(1,d)$, including lengthy special function calculations
involving some divergent integrals that don't have an obvious,
rigorous distributional meaning.
\subsection{Current Work}
The main result of this paper, Theorem \ref{th:main},
rigorously defines the `quantization' of Moser's positive energy
map with well-defined Hilbert spaces. The proof avoids
infinite-dimensional unitary representation theory and only uses
basic definitions and Euler integral formulas for hypergeometric
functions. We carefully define the inverse map, and we explain
the appearance of the fundamental constants like that of
\eqref{eq:invsymm} with the spectral theory of hyperbolic
space. Our approach uniformly handles all dimensions $d \geq 2$,
unlike the methods in \cite{BI66}, which necessitate separate
analyses based on whether $d$ is even or odd.
\par Considering the observations from the previous section, Corollaries \ref{cor:intertwine} and \ref{cor:sym} were established with a certain level of rigor in \cite{BI66,PP66}. To the author's knowledge, Corollary \ref{cor:scat} hasn't been formally stated in the literature, but it follows readily from \cite[\S 4.2]{Y97}. Corollary \ref{cor:group} was certainly known and used by physicists, but again, to the author's knowledge, hasn't been stated nor proved rigorously. Finally, Corollary \ref{cor:formula} appears to be an original formula.
\par In the appendix of this paper, we verify Theorem \ref{th:main} when $f \in
L^2(\mathbb{S}^{d-1})$ is a spherical harmonic. To some extent,
this was done in the appendix of \cite{BI66}. Again, the
computations involved diverging contour integrals and
regularizations that appear without justification. Here, we give an alternative, 
rigorous approach using the special function theory of the
Bessel $J$ and Legendre $P$ functions. It is worth noting that
the results given in this section are strictly \textit{independent} of
our proof of Theorem \ref{th:main}. One may even be able to use
these results together with a density argument to show an alternative
proof of Theorem \ref{th:main}. We refrain from doing so due to
the reliance of nontrivial special function theory.
\subsection{Future Work}
In future work, we plan to study various extensions of Theorem \ref{th:main}. These include extending $L^2(\mathbb{S}^{d-1})$ to $\mathscr{A}'(\mathbb{S}^{d-1})$, analytic functionals (or hyperfunctions) on the sphere (see \cite[\S 5-\S 6]{H74}), and extending $\lambda \in \mathbb{C}$. It is also of interest to try to understand the distributions in Lemma \ref{lem:jared} in the framework of paired or nested Lagrangian distributions (see \cite[\S 5]{dHUV15} and \cite[\S 6.1]{GW23}).
\subsection{Acknowledgments}
This article is part of the Ph.D. thesis of the author at Northwestern University under the guidance of Jared Wunsch and Steve Zelditch. The author thanks Jared Wunsch for advising and support after the passing of Steve Zelditch. The author also thanks email correspondences with Michael E. Taylor and conversations with Xiumin Du and Andr\'as Vasy.
\section{Background}\label{sec:back}

\subsection{Spectral and scattering theory of
  $\widehat{H}_{\hbar}$}
The spectral theory for $\widehat{H}_{\hbar}$ is well-known (see
\cite[Chapter 8, \S 7]{T11part2} for an overview and the informal notes \cite{T22} for interesting conjectures). The spectrum consists of two
disjoint parts: the discrete spectrum lying in the negative real
numbers and the continuous spectrum of $[0,\infty)$. That is,
$$\operatorname{spec}\widehat{H}_{\hbar}=\Big\{
-\frac{1}{2\hbar^2(N+1)^2}:N=0,1,\ldots\Big\}\sqcup \{0\}\sqcup
\Big\{
\frac{1}{2\hbar^2\lambda^2}:\lambda \in (0,\infty)\Big\}. $$
This paper concerns the strictly positive spectrum; that is,
solutions $\psi$ of the equation
\begin{equation}\label{eq:BIG}
\widehat{H}_{\hbar}\psi =E_{\lambda}(\hbar)\psi
\end{equation}
where $E_{\lambda}(\hbar)
\coloneqq\frac{1}{2\hbar^2\lambda^2}$. Writing
\begin{align}\widehat{H}_{\hbar}&=-\frac{\hbar^2}{2}\Big(\partial_r^2+\frac{d-1}{r}\partial_r+\frac{1}{r^2}\Delta_{\mathbb{S}^{d-1}}\Big)-\frac{1}{r}\label{eq:polarH}     
\end{align}
and
separating variables, we obtain explicit generalized
eigenfunctions
\begin{equation}\label{eq:geneig}
\widehat{H}_{\hbar}\psi_{\hbar,\lambda, \ell }^{\boldsymbol{m}}=E_{\lambda}(\hbar)\psi_{\hbar,\lambda,\ell}^{\boldsymbol{m}}
\end{equation}
where \begin{align}
\psi_{\hbar,\lambda,\ell}^{\boldsymbol{m}}(r,\theta) \coloneqq &\
                                                         C_{\lambda,\ell}\Big( \frac{2}{\hbar^2 \lambda}ri\Big)^{\ell}e^{\frac{r}{\hbar^2\lambda}i}{\bf
                                                         M}\Big(\frac{d-1}{2}+\ell-\lambda
                                                         i;\
                                                         d-1+ 2\ell;\
                                                         -\frac{2}{\hbar^2 \lambda}ri
                                                         \Big)Y_{\ell}^{\boldsymbol{m}}(\theta)\label{eq:geneigfun} \\
                                                       =&\
                                                          \frac{C_{\lambda,\ell}}{\Gamma(d-1+2\ell)}\Big(\frac{2}{\hbar^2
                                                          \lambda}ri
                                           \Big)^{-\frac{d-1}{2}}M_{\lambda
                                                          i,
                                                          \frac{d}{2}+\ell-1}\Big(-\frac{2}{\hbar^2
                                                          \lambda}ri\Big)Y_{\ell}^{\boldsymbol{m}}(\theta)
                                                          \nonumber\\
        =&\ C_{\lambda,\ell}\frac{\Gamma(1-\frac{d-1}{2}-\ell+\lambda i )}{\Gamma(\frac{d-1}{2}+\ell-\lambda i)}\Big(\frac{2}{\hbar^2
                                                          \lambda}ri
                                           \Big)^{-\frac{d-1}{2}}L_{-\frac{d-1}{2}-\ell+\lambda i }^{(d-2+2\ell)}\Big(-\frac{2}{\hbar^2 \lambda}ri\Big)Y_{\ell}^{\boldsymbol{m}}(\theta)\nonumber
\end{align}
where $r \in [0,\infty), \theta \in \mathbb{S}^{d-1}$, $C_{\lambda,\ell}$ is a normalization constant to be determined later (see Proposition \ref{prop:main} in the appendix), ${\bf M}$ is Olver's confluent hypergeometric function, $M$ is Whittaker's function, $L$ is the analytically continued generalized Laguerre polynomials (see \cite{O74} for the definitions), and $Y_{\ell}^{\boldsymbol{m}}:\mathbb{S}^{d-1}\to \mathbb{C} $ are the spherical harmonics satisfying
$$-\Delta_{\mathbb{S}^{d-1}}Y_{\ell}^{\boldsymbol{m}}(\theta)=\ell(\ell+d-2)Y_{\ell}^{\boldsymbol{ m}}(\theta) $$
where $\boldsymbol{m}$ denotes the multi-index of  $d-2$
quantum numbers identifying each degenerate harmonic for each
$\ell$ (see \cite[\S 2.2]{Y97} for derivations of
\eqref{eq:geneig}). \par We can construct more generalized
eigenfunctions with scattering theory (see \cite{M95,H12} for 
general overviews of the time-independent and time-dependent theories, respectively, and \cite{J99,M92,V98} for smooth, long-range, time-independent
scattering and \cite[\S 9]{RS79} for time-dependent Coulomb scattering). Following \cite[\S 2.1]{Y97}, we seek to find the so-called
``perturbed plane waves'' which satisfy \eqref{eq:BIG} and are of the form
\begin{equation}\label{eq:ppw}
\psi_{\hbar,\lambda}(x; \theta_0)=c_{\hbar,\lambda} e^{i\frac{1}{\hbar^2\lambda}x\cdot \theta_0 }f(|x|-x\cdot \theta_0)
\end{equation}
for some function $f$ and fixed constants $c_{\hbar,\lambda}\in \mathbb{R}$ and $\theta_0 \in
\mathbb{S}^{d-1}$. Using the fact that
$$\Delta_{\mathbb{S}^{d-1}}\big[g( \theta_0\cdot \theta) \big]=-(d-1)( \theta_0\cdot  \theta)
g'( \theta_0\cdot  \theta )+(1-(\theta_0\cdot
\theta)^2)g''( \theta_0\cdot \theta) $$
for any $g \in C^2(\mathbb{R})$ and $\theta_0 \in \mathbb{S}^{d-1}$ fixed,
we see that $f$ from \eqref{eq:ppw} satisfies
$$2tf''(t)+\Big(d-1-\frac{2}{\hbar^2\lambda}ti\Big)f'(t)+\frac{2}{\hbar^2}f(t)=0. $$
This equation can be written as a special case of the confluent
hypergeometric equation, and we pick the regular solution
distinguished by $f(0)=\Gamma(\frac{d-1}{2})^{-1}$ and $f'(0)=-\frac{2}{\hbar^2(d-1)}$. So $f\equiv {\bf M}(\lambda i,\frac{d-1}{2},\frac{1}{\hbar^2\lambda}i\bullet )$ and \eqref{eq:ppw} is
\begin{equation}\label{eq:ppw2}
\psi_{\hbar,\lambda}(x;\theta_0)=c_{\hbar,\lambda}e^{i
  \frac{1}{\hbar^2\lambda}x\cdot\theta_0 }{\bf
  M}\Big(\lambda i;\ \frac{d-1}{2};\ \frac{1}{\hbar^2\lambda}(|x|-x\cdot \theta_0)i\Big)
\end{equation}
where $c_{\hbar,\lambda}\coloneqq
|\lambda|^{\frac{d}{2}-1}\hbar^{\frac{d}{2}}\sqrt{2\pi}$ (the
value of this constant comes lining up constants in Theorem
\ref{th:main}). \par Note that for the repulsive Coulomb operator \eqref{eq:HamilRep}, we can define the analogous perturbed plane wave by
\begin{equation}\label{eq:future}
\psi_{\hbar,\lambda}^{-}(x;\theta_0)=c_{\hbar,\lambda}e^{-\pi |\lambda|}e^{i
  \frac{1}{\hbar^2\lambda}x\cdot\theta_0 }{\bf
  M}\Big(-\lambda i;\ \frac{d-1}{2};\ \frac{1}{\hbar^2\lambda}(|x|-x\cdot \theta_0)i\Big),
\end{equation}
which is, up to constants, the same as \eqref{eq:ppw2} with a negation of the
first slot of the ${\bf M}$ function. Again, we leave the reader to verify Theorem \ref{th:main2} by adapting the proof of Theorem \ref{th:main}, which we give in \S\ref{sec:good}. \par Now since ${\bf M}$ is entire in all three slots, we have an Euler integral representation given by
\begin{align}
&\psi_{\hbar,\lambda}(x;\theta_0)=c_{\hbar,\lambda}e^{i
  \frac{1}{\hbar^2\lambda}x\cdot\theta_0 }\lim_{\varepsilon \to 0^+}{\bf
  M}\Big(\varepsilon+\lambda i;\ \frac{d-1}{2};\
                                  \frac{1}{\hbar^2\lambda}(|x|-x\cdot
                                  \theta_0)i\Big)\label{eq:BIG0} \\
                                &=\lim_{\varepsilon \to 0^+}\frac{c_{\hbar,\lambda}}{\Gamma(\varepsilon+\lambda i)\Gamma(\frac{d-1}{2}-\varepsilon-\lambda i)}\int_0^1e^{i\frac{t}{\hbar^2\lambda}|x|}e^{-i\frac{t-1}{\hbar^2\lambda} x \cdot \theta_0 }t^{\varepsilon+\lambda i-1}(1-t)^{\frac{d-1}{2}-\varepsilon-\lambda i-1}dt.\label{eq:BIG}
\end{align}
We can trivially bound $\psi_{\hbar,\lambda}(x;\theta_0)$ by a
linear function of $|x|$ with the triangle inequality. Indeed, we have
\begin{align}
&\Big|{\bf
  M}\Big(\varepsilon+\lambda i;\ \frac{d-1}{2};\
                                  \frac{1}{\hbar^2\lambda}(|x|-x\cdot
                                  \theta_0)i\Big)\Big| \leq |c_{\hbar,\lambda}|\bigg|\frac{1}{\Gamma(\varepsilon+\lambda
  i)\Gamma(\frac{d-1}{2}-\varepsilon-\lambda i)}\nonumber \\
                                                       &\qquad \bigg(\int_0^1(e^{i\frac{t}{\hbar^2\lambda}|x|}e^{-i\frac{t}{\hbar^2\lambda}
                 x \cdot
                                                      \theta_0
                                                      }-1)t^{\varepsilon+\lambda i-1}(1-t)^{\frac{d-1}{2}-\varepsilon-\lambda i-1}dt+
                                                       \int_0^1t^{\varepsilon+\lambda i-1}(1-t)^{\frac{d-1}{2}-\varepsilon-\lambda i-1}dt\bigg)
      \bigg|\nonumber \\
  &\leq
    \frac{2|c_{\hbar,\lambda}||x|}{|\Gamma(\varepsilon+\lambda
    i)\Gamma(\frac{d-1}{2}-\varepsilon-\lambda
    i)|}\int_0^1t^{\varepsilon}(1-t)^{\frac{d-1}{2}-\varepsilon
    -1}dt+\frac{|c_{\hbar,\lambda}|}{\Gamma( \frac{d-1}{2})}\nonumber \\
  &\leq  \frac{2|c_{\hbar,\lambda}||x|}{|\Gamma(\varepsilon+\lambda i)\Gamma(\frac{d-1}{2}-\varepsilon-\lambda i)|}\frac{\Gamma(\varepsilon+1)\Gamma(\frac{d-1}{2}-\varepsilon)}{\Gamma(\frac{d+1}{2})}+\frac{|c_{\hbar,\lambda}|}{\Gamma( \frac{d-1}{2})}=A_{\hbar,\lambda}|x|+B_{\hbar,\lambda},\label{eq:bound2}
\end{align}
for $\varepsilon<1$ and constants $A_{\hbar,\lambda},B_{\hbar,\lambda}>0$, where we have split the complex exponential in the integrand as
$e^{iy}=e^{iy}-1+1$ and have used $|e^{iy}-1|\leq |y|$. This
along with the dominated convergence theorem shows that the
limit in \eqref{eq:BIG} converges in tempered distributions
$\mathcal{S}'(\mathbb{R}^d)$. Note that one can refine
\eqref{eq:bound2} using asymptotics of ${\bf M}$ (see
\cite[(2.4)]{Y97}). Indeed, we have the Poincar\'{e} series
\begin{align*}
{\bf M}\Big(\lambda i , \frac{d-1}{2},i r\Big) &\sim \frac{e^{i
  r}(ir)^{\lambda i -\frac{d-1}{2}}}{\Gamma(\lambda
  i)}\sum_{k=0}^{\infty}\frac{(1-\lambda
  i)_{k}(\frac{d-1}{2}-\lambda i)_{k}}{k!}(ir)^{-k}\\
  &\qquad \qquad \qquad +\frac{e^{-\pi \lambda }(ir)^{-\lambda i}}{\Gamma(\frac{d-1}{2}-\lambda i)}\sum_{k=0}^{\infty}\frac{(\lambda i)_k(\lambda i -\frac{d-1}{2}+1)_k}{k!}(-ir)^{-k}
\end{align*}
as $r\to +\infty$, where $(a)_k \coloneqq
\frac{\Gamma(a+k)}{\Gamma(a)}$ is the rising Pochhammer symbol
(see
\cite[\href{https://dlmf.nist.gov/13.7E2}{(13.7.2)}]{NIST:DLMF}). In particular, this shows for fixed $\hbar>0$ and $\lambda \neq 0$ we have
\begin{equation}\label{eq:linfinity}
\psi_{\hbar,\lambda}(\bullet;\bullet) \in L^{\infty}(\mathbb{R}^d \times \mathbb{S}^{d-1}).
\end{equation}
With the functions $\psi_{\hbar,\lambda}(\bullet;\bullet)$, we can define the Hilbert spaces $\mathscr{H}_{\widehat{H}_{\hbar}}(\hbar,\lambda)$ (see \eqref{eq:Hilbert}). The norm of these Hilbert spaces is indeed positive definite, as proved at the end of step 2 of \S \ref{sec:good}. The functions $\psi_{\hbar,\lambda}(\bullet;\bullet)$ also define the Poisson operators $P_{\widehat{H}_{\hbar}}(\hbar,\lambda)$ and the scattering matrices $S_{\widehat{H}_{\hbar}}(\hbar,\lambda)$ (see \eqref{eq:poissonops},\eqref{eq:scatmatrices}, respectively). For concreteness, note that
$P_{\widehat{H}_{\hbar}}(\hbar,\lambda)[Y_{\ell}^{\boldsymbol{m}}]$
is \eqref{eq:geneigfun}, derived in Proposition \ref{prop:main}.
\begin{figure}[!htbp]
   \begin{minipage}{0.5\linewidth}
     \centering
     \includegraphics[width=\linewidth]{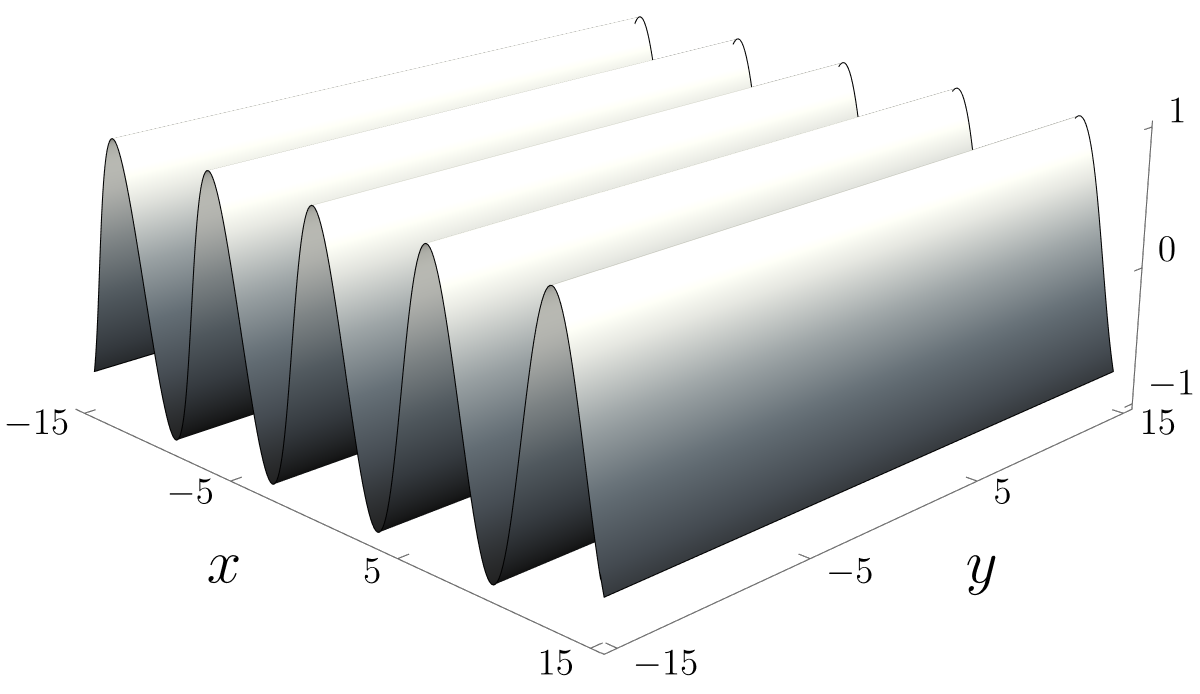}
     \captionsetup{width=0.95\linewidth}
     \caption{Plot of $\Re(e^{i x})$ when $d=2$, $\hbar=\lambda=1$}\label{fig:ppweuc}
   \end{minipage}\hfill
   \begin{minipage}{0.5\linewidth}
     \centering
     \includegraphics[width=\linewidth]{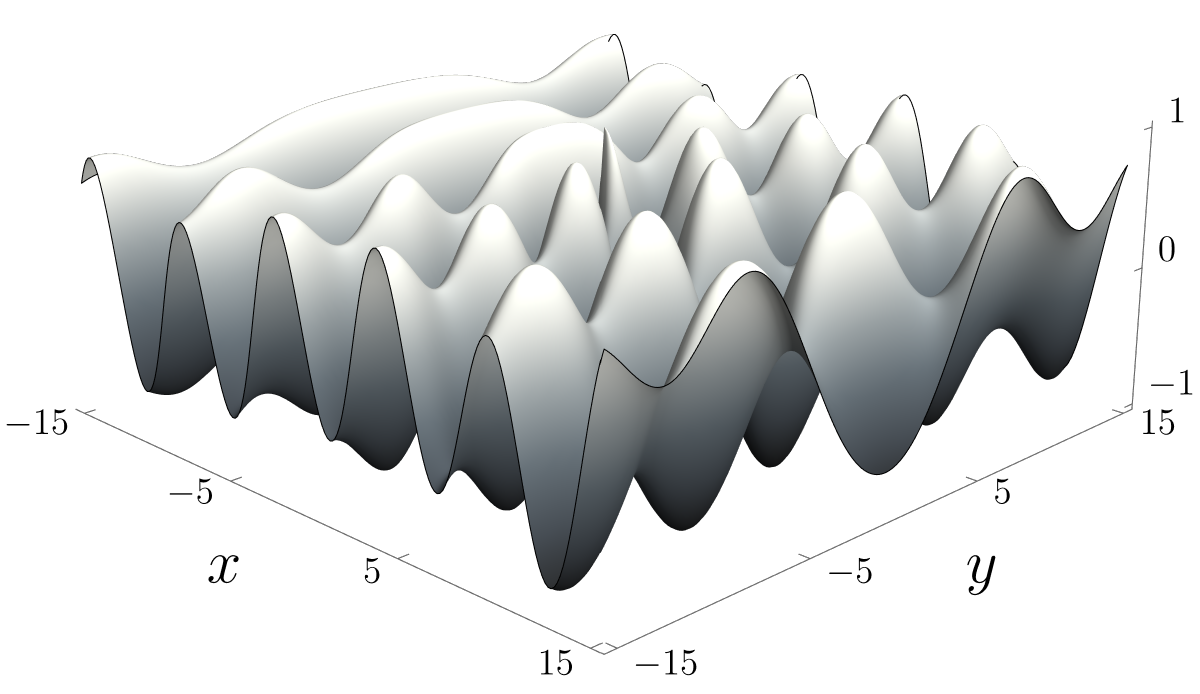}
     \captionsetup{width=0.95\linewidth}
     \caption{Plot of $\frac{1}{\sqrt{2}}\Re(\psi_{1,1}((x,y);(1,0)))$ when $d=2,\hbar=\lambda=1$}\label{fig:ppw}
   \end{minipage}
\end{figure}

\subsection{Spectral and scattering theory of
  $\Delta_{\mathbb{H}^d}$}\label{sec:hyp}
The aim of this section is to give background of analysis on $\mathbb{H}^d$. The standard reference is \cite{H84}, and two nice references focused on $\mathbb{H}^d$ are \cite{B92} and the more recent \cite{GL23}. Define
hyperbolic space by
\begin{equation}\label{eq:hypomeg}
\mathbb{H}^d\coloneqq \{u \in \mathbb{R}^{d}: |u|<1\}
\end{equation}
with the Riemannian metric
$ds^2=\frac{4}{1-|u|^2}(du_1^2+\cdots+du_d^2)$. For any $u \in
\mathbb{H}^d$ and $\theta_0 \in \partial
\mathbb{H}^d=\mathbb{S}^{d-1}$, define $\langle u,\theta_0
\rangle$ to be the (signed) hyperbolic distance from the origin to the
horocycle connecting $u$ and $\theta_0$ (the distance is
negative if the origin lies in the interior of the
horocycle). Define the plane waves
\begin{equation}\label{eq:hypplane}
e_{\lambda}(u;\theta_0)=e^{(\frac{d-1}{2}-\lambda i)\langle u, \theta_0 \rangle}.
\end{equation}
It can be shown that they satisfy
\begin{equation*}
-\Delta_{\mathbb{H}^d}e_{\lambda}(\bullet;\theta_0)=\Big(\lambda^2+\Big(\frac{d-1}{2}\Big)^2\Big)e_{\lambda}(\bullet;\theta_0),
\end{equation*}
and with the
law of cosines, we have
\begin{equation}\label{eq:hypplane2}
e_{\lambda}(u;\theta_0)=\bigg(\frac{1-|u|^2}{|u-\theta_0|^2}\bigg)^{\frac{d-1}{2}-\lambda i}.
\end{equation}
When $d=2$ and $\lambda=i/2$, \eqref{eq:hypplane2} is the Euclidean Poisson kernel (this relation, however, only holds when $d=2$; see \cite{T18}).
With the functions $e_{\lambda}(\bullet;\bullet)$, we can define the Hilbert spaces $\mathscr{H}_{\mathbb{H}^d}(\lambda)$ (see \eqref{eq:Hilbert}). The norm of these Hilbert spaces is indeed positive definite for $\lambda$ not a pole of the denominator of the Harish-Chandra $\mathfrak{c}$ function, as shown in \cite[Introduction, Proposition 4.8]{H84} for $d=2$ and \cite[Chapter III, \S 5]{H70} in general (in our case, this condition is satisfied when $\frac{d-1}{2}+\lambda i \not\in \mathbb{Z}_{\leq 0}$, which is clear since $\lambda \in \mathbb{R}$). The functions $e_{\lambda}(\bullet;\bullet)$ also define the Poisson operators $P_{\mathbb{H}^d}(\lambda)$ and the scattering matrices $S_{\mathbb{H}^d}(\lambda)$ (see \eqref{eq:poissonops},\eqref{eq:scatmatrices}, respectively). For concreteness, note that
$P_{\mathbb{H}^d}(\lambda)[Y_{\ell}^{\boldsymbol{m}}]$ is
\eqref{eq:legend}, derived in Proposition \ref{prop:main}.
\begin{figure}[!htbp]
   \begin{minipage}{0.5\linewidth}
     \centering
     \includegraphics[width=\linewidth]{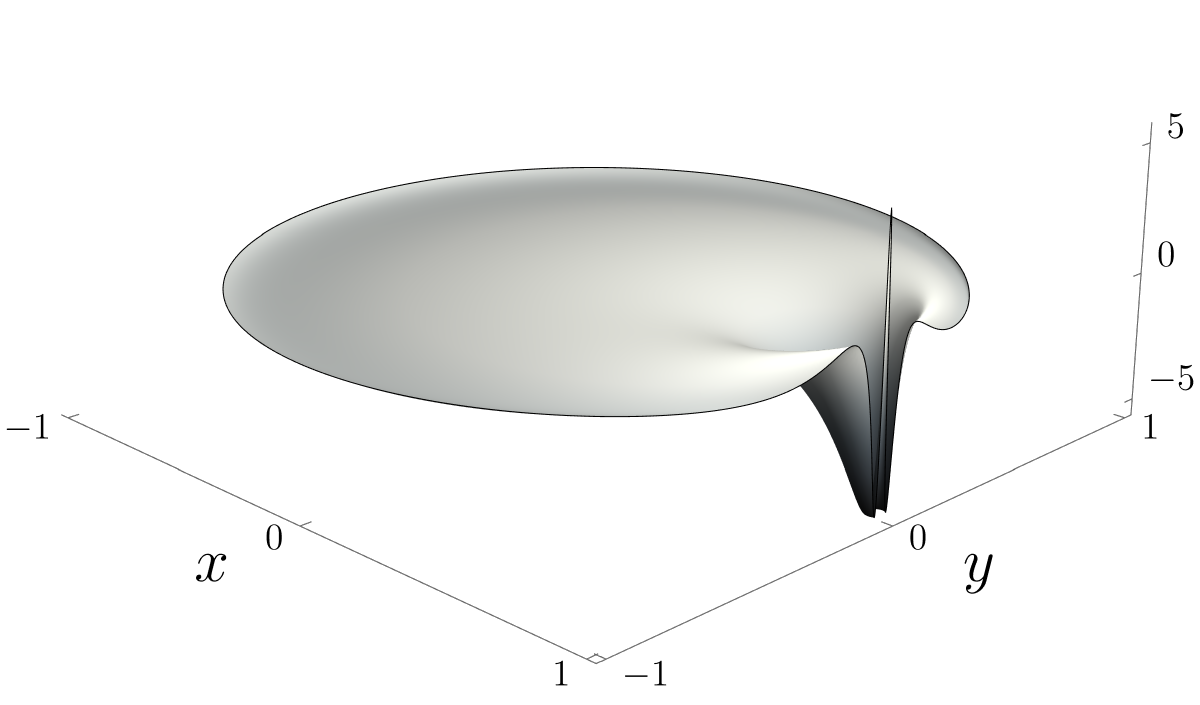}
     \captionsetup{width=0.9\linewidth}
     \caption{Plot of $\Re(e_1((x,y);(1,0)))$ when $d=2$, $\hbar=\lambda=1$.} \label{fig:ppwhyp}
   \end{minipage}\hfill
   \begin{minipage}{0.5\linewidth}
     \centering
     \includegraphics[width=\linewidth]{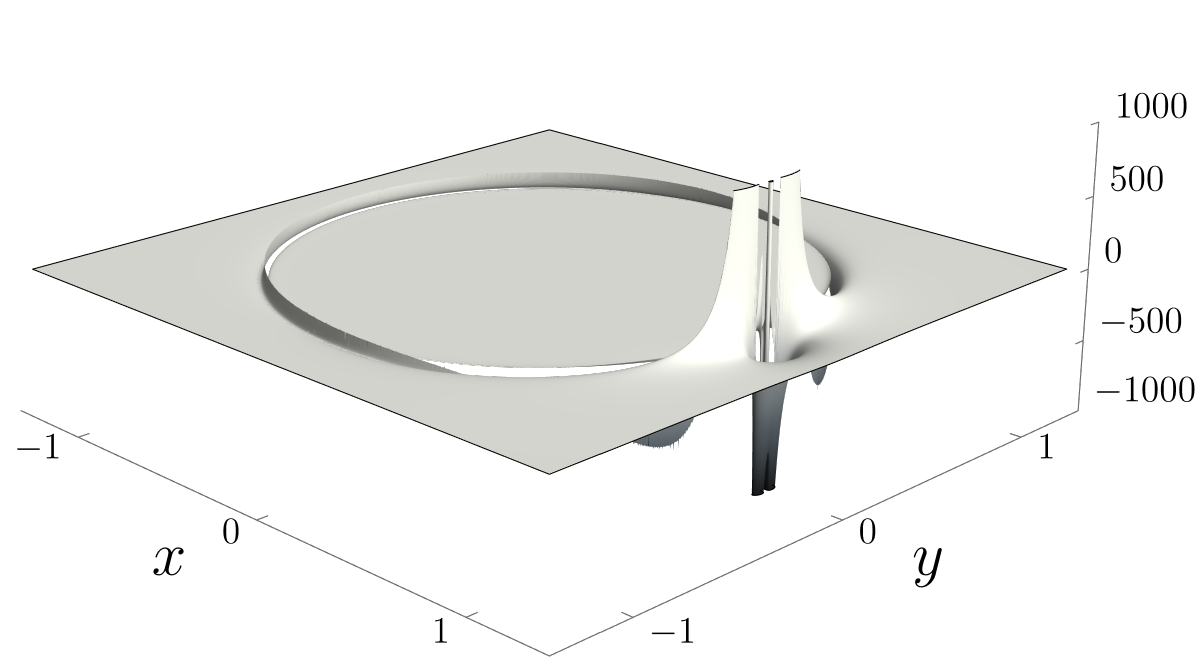}
     \captionsetup{width=\linewidth}
     \caption{Plot of $(\Re \circ \mathcal{F})[\psi_{1,1}(\bullet;(1,0))](x,y)$ when $d=2,\hbar=\lambda=1$.}\label{fig:ft}
   \end{minipage}
\end{figure}

\subsection{Classical mechanical Moser
  map}\label{sec:moser}
As noted in \S\ref{sec:prior}, the positive energy Hamiltonian orbits of the attractive Coulomb Hamiltonian 
$$H^{+}(x,\xi) \coloneqq \frac{|\xi|^2}{2}- \frac{1}{|x|},$$
also known as Kepler orbits, follow planar hyperbolic position-space trajectories, as well as straight-line collisions into the origin in finite time. Here, we explain Moser's reflection condition regularization. A similar regularization can be applied to the repulsive Coulomb Hamiltonian 
$$H^{-}(x,\xi) \coloneqq \frac{|\xi|^2}{2}+ \frac{1}{|x|},$$
where the straight line orbits are reflected at `infinity' (note
that there are no collision orbits in the repulsive setting). We
describe both the attractive and repulsive settings with the
two-sheeted hyperbolic space below.
\par Define the punctured, two-sheeted hyperbolic
space $\widetilde{\mathbb{H}}^d$ by
\begin{equation}\label{eq:hypomeg}
\widetilde{\mathbb{H}}^d\coloneqq \{u \in \mathbb{R}^{d}: |u|\neq 0,1\}
\end{equation}
with the hyperbolic metric
$ds^2=\frac{4}{|1-|u|^2|}(du_1^2+\cdots+du_d^2).$
Let $\iota: \widetilde{\mathbb{H} }^d \to \widetilde{\mathbb{H}
}^d$ be defined as $\iota(u) \coloneqq u/|u|^2$, and one
can easily compute that the pullback
$\iota^*:T^*\widetilde{\mathbb{H} }^d \to T^*
\widetilde{\mathbb{H} }^d$ is
$$\iota^*(u,\eta)=\Big(\frac{u}{|u|^2},\frac{|1-|u|^2|}{2|u|^2}(\eta|u|^2-2u (\eta \cdot
u))\Big)$$
where we have identified the domain $T^*\widetilde{\mathbb{H} }^d=T\widetilde{\mathbb{H} }^d$ with the musical isomorphism induced from the \textit{Euclidean} metric on $\widetilde{\mathbb{H} }^d$, and the codomain $T^*\widetilde{\mathbb{H} }^d=T\widetilde{\mathbb{H} }^d$ with the musical isomorphism induced from the \textit{hyperbolic} metric on $\widetilde{\mathbb{H} }^d$. Furthermore, in both the domain and codomain, we trivialize $T \widetilde{\mathbb{H}}^d=\widetilde{\mathbb{H}}^d \times \mathbb{R}^d$ by viewing $\widetilde{\mathbb{H}}^d \subset \mathbb{R}^d$.
\begin{defn}
Let $E>0$ and $p_0 \coloneqq \sqrt{2E}$. Identifying
$\mathbb{R}^d \times \{|\xi|\neq 0,p_0\} \subset T^*\mathbb{R}^d$,
we define the Moser map $\mathcal{M}_E:\mathbb{R}^d \times \{|\xi|\neq 0,p_0\} \to T^*\widetilde{\mathbb{H}}^d$ by
$$\mathcal{M}_E \coloneqq \iota^* \circ R_{-\pi/2}\circ S_{p_0}
\circ \mathcal{D}_{p_0}$$
where $\mathcal{D}_{p_0}(x,\xi) \coloneqq (p_0x, \frac{1}{p_0}\xi)$ is the symplectic
dilation by $p_0$, $S_{p_0}(x,\xi) \coloneqq (p_0x,\xi)$ is the nonsymplectic dilation in position by $p_0$, and $R_{-\pi/2} \coloneqq
(\xi,-x)$ is the symplectic rotation by $-\pi/2$. Explicitly,
\begin{align*}
\mathcal{M}_{E}(x,\xi)&=\Big(p_0 \frac{\xi}{|\xi|^2}, \frac{|p_0^2-|\xi|^2|}{2|\xi|^2} (-x|\xi|^2+2\xi (x \cdot
                        \xi))\Big),\\
  \mathcal{M}_{E}^{-1}(u,\eta)&=\Big(\frac{2}{p_0^2|1-|u|^2|}(-\eta |u|^2+2u(u \cdot \eta)) ,p_0\frac{u}{|u|^2}\Big).
\end{align*}
\end{defn}
\begin{rem}
It is easy to see that $\mathcal{M}_E$ pulls back the canonical symplectic form of
$T^*\widetilde{\mathbb{H} }^d$ to $p_0\ dx \wedge
d\xi$. Additionally, we see that if
$(u,\eta)=\mathcal{M}_E(x,\xi)$, then $|u|=p_0/|\xi|$ and $|\eta|=|x|\frac{||\xi|^2-p_0^2|}{2}$ and hence
$$ H^{+}(x,\xi)=E \iff 0<|u|<1 \text{ and }|\eta|=1. $$
That is, the Moser map takes the energy surface $\Sigma_E^+ \coloneqq \{H^+(x,\xi)=E\}$ of the attractive Coulomb Hamiltonian (i.e. the Kepler Hamiltonian) to the unit cotangent bundle of the (punctured) Poincar\'{e} unit ball model of hyperbolic space $\mathbb{H}^d$. Additionally, the Moser map takes the energy surface $\Sigma_E^{-} \coloneqq \{H^{-}(x,\xi)=E\}$ of the repulsive Coulomb Hamiltonian $H^{-}(x,\xi)=E$ to the unit cotangent bundle of the other hyperbolic space of $\widetilde{\mathbb{H}}^d$ (that is, the Poincar\'{e} model of the complement of the unit ball). See Figure \ref{fig:2} for a plot of the level sets of $H^{+}$ and $H^{-}$ when $E=1/2$. \par Up to a reparametrization of time, the Moser map $\mathcal{M}_E$ transforms the Coulomb Hamiltonian flow on the energy surface $H^{\pm}(x, \xi) = E$ onto the cogeodesic flow on $T_1^*\widetilde{\mathbb{H}}^d$ parametrized by arc length. In particular, the geodesics ending at the origin of $\widetilde{\mathbb{H}}^d$ are mapped to the collision orbits of $H^{-}$, and the straight-line geodesics of $\widetilde{\mathbb{H}}^d$ going to infinity are mapped to the straight-line Hamiltonian orbits of $H^{+}$ going to infinity. The time reparametrization can be found explicitly. Indeed, 
if $\gamma(t)=(x(t),\xi(t)) \in T^*\mathbb{R}^d$ is a Hamiltonian orbit on the energy surface $H^{\pm}(x,\xi)=E$, then $\varphi(s)=(u(s),\eta(s)) \coloneqq \mathcal{M}_E(\gamma(t(s)) \in T_1^* \widetilde{\mathbb{H} }^d$ is a 
(co)geodesic on $\widetilde{\mathbb{H} }^d$ parametrized by arc length $s$ where
$t(s)$ satisfies
\begin{equation}\label{eq:timechange}
\frac{dt}{ds}=\frac{|x(t(s))|}{p_0}=\frac{2|u(s)|^2}{p_0^3\big|1-|u(s)|^2\big|}=\frac{1}{p_0^3}|u(s)| \cdot\lVert u(s) \rVert_{u(s)},\quad t(0)=0,
\end{equation}
and $\lVert \bullet \rVert$ is the norm induced by the
hyperbolic metric on $\widetilde{\mathbb{H}}^d$. To prove this,
one can adapt the negative energy computations in
\cite{M70,HdL12} to the positive energy case. Curiously,
\eqref{eq:timechange} is a generalization of Kepler's third law
for hyperbolic orbits (see \eqref{eq:kepler3} for the original, negative energy statement). Namely, if $\varphi(s)\in  T_1^*
\widetilde{\mathbb{H} }^d$ is a cogeodesic parametrized by arc
length, then the time parameter $t_E(s)$ of the Hamiltonian
orbit $\mathcal{M}_E^{-1}(\varphi)$ is equal to the time
parameter $t_{1/2}(s)$ of the Hamiltonian orbit
$\mathcal{M}_{1/2}^{-1}(\varphi)$ up to a factor of $1/p_0^3$.
\end{rem}
\begin{figure}[!htbp]
    \centering
    \includegraphics[width=0.75\linewidth]{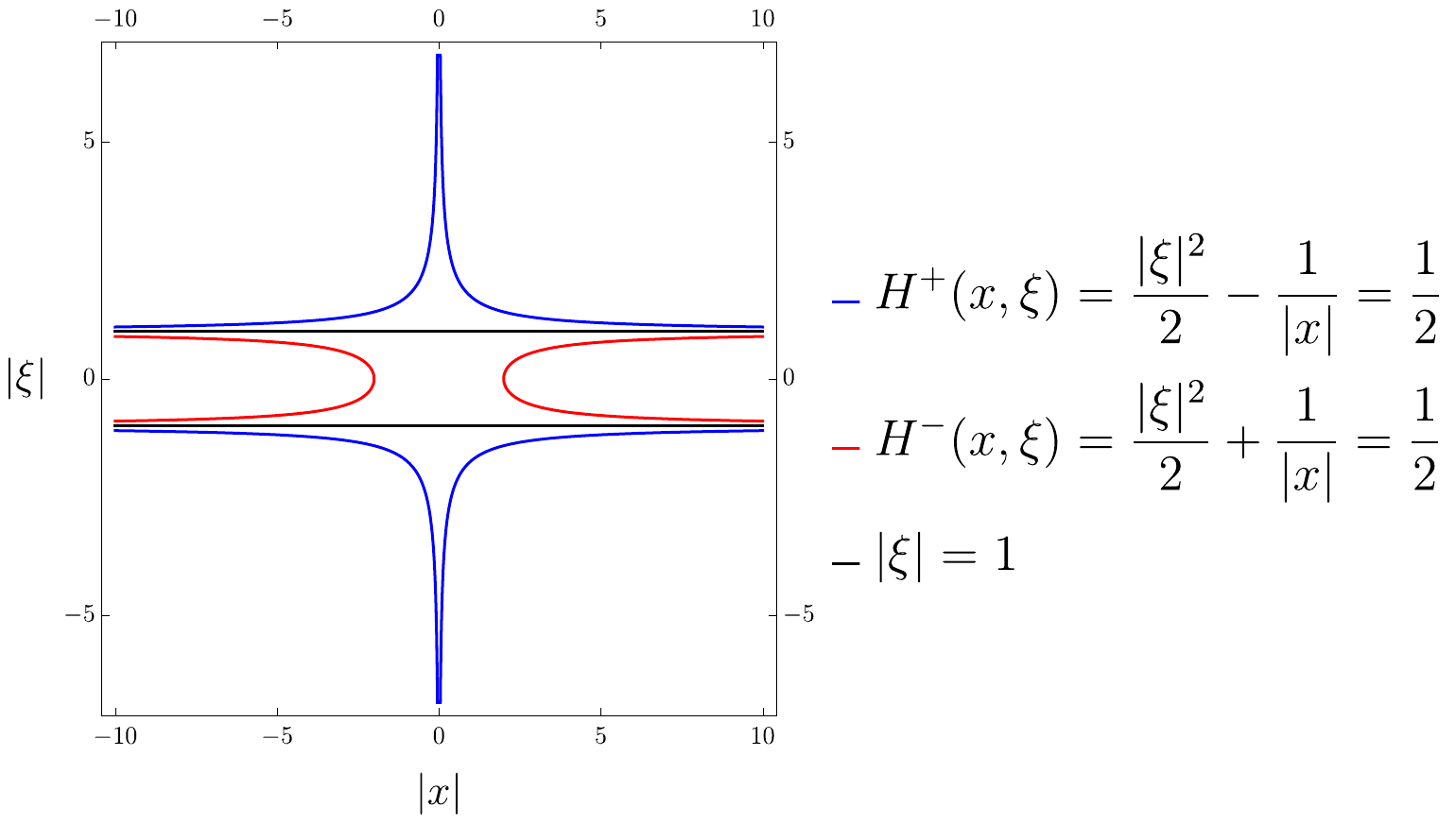} %
    \caption{Plot of the level set $E=1/2$ of the attractive Coulomb Hamiltonian $H^+$ and the repulsive Coulomb Hamiltonian $H^{-}$}\label{fig:2}%
    \end{figure}
    \section{Proof of Theorem \ref{th:main}}\label{sec:good}
    The proof can be broken down into three steps: the first and most pivotal step shows $\mathcal{V}_{\hbar,\lambda}$ maps the perturbed Coulomb plane waves
$\psi_{\hbar,\lambda}$ to the hyperbolic plane waves
$e_{\lambda}$ (i.e.,
$\mathcal{V}_{\hbar,\lambda}[\psi_{\hbar,\lambda}(\bullet;\theta_0)]=e_{\lambda}(\bullet;\theta_0)$);
the second step shows $\mathcal{V}_{\hbar,\lambda}$ commutes with
integrating in $\theta_0$ against $f \in L^2(\mathbb{S}^{d-1})$; and the third step defines the
inverse map $\mathcal{V}_{\hbar,\lambda}^{-1}$ and proves it
indeed is the inverse of  $\mathcal{V}_{\hbar,\lambda}$. Before we begin with step 1, we prove a very important lemma:
\begin{lem}\label{lem:jared} For $\lambda \neq 0, \hbar>0,\theta_0 \in \mathbb{S}^{d-1}$ and
$\psi_{\hbar,\lambda}(\bullet ;\theta_0),\widehat{\mathcal{D}}_{\frac{1}{\hbar\lambda}},\mathcal{F}_{\hbar} $ defined in
\eqref{eq:ppw2},\eqref{eq:mainops},
\begin{align}
&(\widehat{\mathcal{D}}_{\frac{1}{\hbar\lambda}}\circ
  \mathcal{F}_{\hbar})[\psi_{\hbar,\lambda}(\bullet;\theta_0)](\xi)\nonumber \\
  &=\lim_{\varepsilon \to
  0^+} \frac{2^{\frac{d+1}{2}}(\varepsilon + \lambda
  i)}{|\lambda|}\frac{
  \varepsilon^2\Big(1+\frac{(\frac{d-1}{2}-\varepsilon- \lambda
  i)(|\xi|^2-1+\varepsilon^4-2\operatorname{sgn} (\lambda)\varepsilon^2i 
  )}{(\varepsilon+ \lambda i)(|\xi-\theta_0|^2+ \varepsilon^4)}\Big)- \operatorname{sgn}(\lambda)i}{(|\xi|^2-1+\varepsilon^4-2\operatorname{sgn}(\lambda) \varepsilon^2i)^{1+ \varepsilon+\lambda i}(|\xi-\theta_0|^2+ \varepsilon^4)^{\frac{d-1}{2}-\varepsilon- \lambda i}}.\label{eq:allie}
\end{align}
When $\xi$ is away from $\theta_0$, \eqref{eq:allie} simplifies to 
\begin{equation}\label{eq:jared}
(\widehat{\mathcal{D}}_{\frac{1}{\hbar\lambda}}\circ \mathcal{F}_{\hbar})[\psi_{\hbar,\lambda}(\bullet;\theta_0)](\xi)=\frac{2^{\frac{d+1}{2}}}{(|\xi|^2-1- \operatorname{sgn}(\lambda) 0 i)^{1+\lambda i}|\xi-\theta_0|^{d-1- 2\lambda i}}.
\end{equation}
The distribution \eqref{eq:allie} is the unique distribution equaling \eqref{eq:jared} away from $\theta_0$ in the intersection of Sobolev spaces $H^{-\frac{d}{2}-0}(\mathbb{R}^d) \coloneqq \bigcap_{\varepsilon >0} H^{-\frac{d}{2}-\varepsilon}(\mathbb{R}^d) $, up to adding a $c \delta_{\theta_0}$ term.
\end{lem}
\begin{rem}
From \eqref{eq:jared}, we see that the singular support of $(\widehat{\mathcal{D}}_{\frac{1}{\hbar\lambda}}\circ
  \mathcal{F}_{\hbar})[\psi_{\hbar,\lambda}(\bullet;\theta_0)]$
  is $\mathbb{S}^{d-1}$. The singularities are comprised of
  rather `mild' $\pm i0$ singularities at points on the sphere
  away from $\theta_0$ and a `strong,' delta-like singularity at
  $\theta_0$. Evidently, like paired Lagrangian distributions,
  these two singularities are `linked' together are not easily
  separated (see Figure \ref{fig:ft}). It is of interest to
  understand the regularization \eqref{eq:allie} of
  \eqref{eq:jared} in terms of Lagrangian distributions (defined
  in \cite[\S 5]{dHUV15} and \cite[\S 6.1]{GW23}), but the
  current uniqueness statement \textit{does} narrow down the
  distribution to some extent.
  \par Regardless, the distribution compares nicely to $(2\pi)^{\frac{d}{2}}\delta_{\theta_0}$, the Fourier transform of the unperturbed plane wave $e^{i x \cdot \theta_0}$ in $\mathbb{R}^d$.
\end{rem}
\begin{proof}[Proof of Lemma \ref{lem:jared}]
Suppose $\lambda>0$, as we can recover the negative case by
using
$\overline{\psi_{\hbar,\lambda}}=\psi_{\hbar,-\lambda}$. Commuting the dilation
with the Fourier transform, we have
\begin{equation*}
(\widehat{\mathcal{D} }_{\frac{1}{\lambda \hbar}}\circ \mathcal{F}_{\hbar})[\psi_{\hbar,\lambda}]=\frac{\sqrt{2\pi}}{\lambda}\mathcal{F}[e^{i \bullet \cdot \theta_0}f]=\frac{\sqrt{2\pi}}{\lambda}\mathcal{F}[f](\bullet -\theta_0)
\end{equation*}
where $f(x) \coloneqq {\bf
  M}\big(\lambda i;\ \frac{d-1}{2};\ (|x|-x\cdot
\theta_0)i\big)$. We have
\begin{align}
f(x)&=\lim_{\varepsilon \to
      0^+}\frac{1}{\Gamma(\varepsilon+\lambda
      i)\Gamma(\frac{d-1}{2}-\varepsilon-\lambda
      i)}\int_0^1e^{it|x|}e^{-i t x \cdot
      \theta_0}t^{\varepsilon+\lambda
      i-1}(1-t)^{\frac{d-1}{2}-\varepsilon-\lambda i-1}dt\label{eq:bb}\\
  &=\lim_{\varepsilon \to 0^+}\frac{1}{\Gamma(\varepsilon+\lambda i)\Gamma(\frac{d-1}{2}-\varepsilon-\lambda i)}\int_0^1e^{(-\varepsilon^2+it)|x|}e^{-itx \cdot \theta_0}t^{\varepsilon+\lambda i-1}(1-t)^{\frac{d-1}{2}-\varepsilon-\lambda i-1}dt\label{eq:holiday}
\end{align}
As noted proceeding \eqref{eq:bound2}, the \eqref{eq:bb} converges in
$\mathcal{S}'(\mathbb{R}^d)$. Additionally, \eqref{eq:holiday} converges in
$\mathcal{S}'(\mathbb{R}^d)$. Indeed, subtracting the two expressions inside the limits when paired with $\varphi \in \mathcal{S}(\mathbb{R}^d)$, we
have
\begin{align*}
&\bigg|\frac{1}{\Gamma(\varepsilon+\lambda
  i)\Gamma(\frac{d-1}{2}-\varepsilon-\lambda
  i)}\int_{\mathbb{R}^d}\int_0^1(e^{-\varepsilon^2|x|}-1)e^{it(|x|-x
  \cdot \theta_0)}t^{\varepsilon+\lambda
  i-1}(1-t)^{\frac{d-1}{2}-\varepsilon-\lambda i-1}dt
  \varphi(x)dx\bigg|\\
  & \leq \frac{1}{|\Gamma(\varepsilon+\lambda
  i)\Gamma(\frac{d-1}{2}-\varepsilon-\lambda
  i)|}\int_{\mathbb{R}^d}\int_0^1t^{\varepsilon-1}(1-t)^{\frac{d-1}{2}-\varepsilon-1}dt|e^{-\varepsilon^2|x|}-1||\varphi(x)|
dx\\
  &=\frac{\Gamma(\varepsilon)\Gamma(\frac{d-1}{2}-\varepsilon)}{\Gamma(\frac{d-1}{2})|\Gamma(\varepsilon+\lambda
  i)\Gamma(\frac{d-1}{2}-\varepsilon-\lambda
    i)|}\int_{\mathbb{R}^d}|e^{-\varepsilon^2|x|}-1||\varphi(x)|dx=O(\varepsilon),
\end{align*}
since $|e^{-\varepsilon^2|x|}-1|\leq \varepsilon^2|x|$ and $\Gamma(\varepsilon)=O(\varepsilon^{-1})$. The Fourier transform in $x$ therefore commutes with the limit in \eqref{eq:holiday}, and we can additionally swap the Fourier transform and
the integral by the Fubini-Tonelli theorem after pairing with $\varphi \in \mathcal{S}(\mathbb{R}^d)$. Using the identity (see \cite[pp. 247]{T11})\begin{align*}
  \mathcal{F}(e^{-(\varepsilon^2-it)|\bullet|})(|\xi|)&=\frac{2^{\frac{d}{2}}\Gamma(\frac{d+1}{2})}{\sqrt{\pi}}\frac{\varepsilon^2-it}{(|\xi|^2+(\varepsilon^2-it)^2)^{\frac{d+1}{2}}},                          
                                                                                                            \end{align*}
we see
\begin{align*}
&\mathcal{F}[f](\xi)=\lim_{\varepsilon \to
  0^+}\frac{2^{\frac{d}{2}}\Gamma(\frac{d+1}{2})}{\sqrt{\pi}\Gamma(\varepsilon+\lambda
  i)\Gamma(\frac{d-1}{2}-\varepsilon-\lambda
  i)}\int_0^1\frac{(\varepsilon^2-it)t^{\varepsilon+\lambda
  i-1}(1-t)^{\frac{d-1}{2}-\varepsilon-\lambda
  i-1}}{(|\xi+t \theta_0|^2+(\varepsilon^2
  -it)^2)^{\frac{d+1}{2}}}dt\\
  &=\lim_{\varepsilon \to
  0^+}\frac{2^{\frac{d}{2}}\Gamma(\frac{d+1}{2})}{\sqrt{\pi}\Gamma(\varepsilon+\lambda
  i)\Gamma(\frac{d-1}{2}-\varepsilon-\lambda
  i)}\int_0^1\frac{(\varepsilon^2-it)t^{\varepsilon+\lambda
  i-1}(1-t)^{\frac{d-1}{2}-\varepsilon-\lambda
  i-1}}{(|\xi|^2+\varepsilon^4+2t (\xi \cdot \theta -i
    \varepsilon^2))^{\frac{d+1}{2}}}dt\\
   &=\lim_{\varepsilon \to
  0^+}\frac{2^{\frac{d}{2}}\Gamma(\frac{d+1}{2})}{\sqrt{\pi}\Gamma(\varepsilon+\lambda
  i)\Gamma(\frac{d-1}{2}-\varepsilon-\lambda
  i)}\int_0^1\frac{(\varepsilon^2-it)t^{\varepsilon+\lambda
  i-1}(1-t)^{\frac{d-1}{2}-\varepsilon-\lambda
  i-1}dt}{\big((|\xi+\theta_0|^2-1+\varepsilon^4-2i \varepsilon^2 )t + (|\xi|^2+ \varepsilon^4)(1-t)\big)^{\frac{d+1}{2}}}
\end{align*}
in distribution. We now take advantage of the Feynman
parametrization identity
\begin{equation}\label{eq:feynman}
\frac{1}{z^{\alpha}w^{\beta}}=\frac{\Gamma(\alpha+\beta)}{\Gamma(\alpha)\Gamma(\beta)}\int_0^1\frac{t^{\alpha-1}(1-t)^{\beta-1}}{\big(zt+w(1-t)\big)^{\alpha+\beta}}dt,\quad \begin{array}{l}\Re \alpha,\Re \beta>0,\\ 0 \not\in \operatorname{conv}(z,w)\subset \mathbb{C},\end{array},
\end{equation}
which can be used to also show
\begin{equation*}\label{eq:feynman2}
\frac{1}{z^{\alpha}w^{\beta}}\bigg(1+\frac{\beta z}{(\alpha-1)w}\bigg)=\frac{\Gamma(\alpha+\beta)}{\Gamma(\alpha)\Gamma(\beta)}\int_0^1\frac{t^{\alpha-2}(1-t)^{\beta-1}}{\big(zt+w(1-t)\big)^{\alpha+\beta}}dt,\ \begin{array}{l}\Re \alpha>1,\Re \beta>0,\\ 0 \not\in \operatorname{conv}(z,w)\subset \mathbb{C},\end{array},
\end{equation*}
(for \eqref{eq:feynman}, see \cite[14a, 15a]{F49} for $\alpha, \beta \in \mathbb{Z}_{>0}$. The general statement can be found in \cite[(3.197) 4.]{GR07}, but it follows quickly from equating the series and integral definitions of the ${}_2F_1$ functions). We then have
\begin{equation}\label{eq:hell}
\mathcal{F}[f](\xi) =\lim_{\varepsilon \to
  0^+} \frac{2^{\frac{d}{2}}(\varepsilon + \lambda
  i)}{\sqrt{\pi}}\frac{
  \varepsilon^2\Big(1+\frac{(\frac{d-1}{2}-\varepsilon- \lambda
  i)(|\xi+\theta_0|^2-1+\varepsilon^4-2i \varepsilon^2
  )}{(\varepsilon+ \lambda i)(|\xi|^2+ \varepsilon^4)}\Big)-i}{(|\xi+\theta_0|^2-1+\varepsilon^4-2i \varepsilon^2 )^{1+ \varepsilon+\lambda i}(|\xi|^2+ \varepsilon^4)^{\frac{d-1}{2}-\varepsilon- \lambda i}}.
\end{equation}
The singular support of this distribution is at the sphere centered at $-\theta_0$ of radius 1, and it is easy to show that if $\varphi \in \mathcal{S}(\mathbb{R}^d)$ where $0\not\in \operatorname{supp}(\varphi)$ we have the distributional pairing
\begin{equation*}
\big(\mathcal{F}[f],\varphi\big)=\bigg(\frac{2^{\frac{d}{2}}\lambda}{\sqrt{\pi}}\frac{1}{(|\xi+\theta_0|^2-1- 0 i)^{1+\lambda i}|\xi|^{d-1- 2\lambda i}},\varphi \bigg),
\end{equation*}
as desired. For the uniqueness statement, since the distribution is the Fourier transform of a function in $L^{\infty}$ (by \eqref{eq:linfinity}), the distribution is automatically in $H^{-\frac{d}{2}-0}$. If we have another distribution in $H^{-\frac{d}{2}-0}$ equal to \eqref{eq:jared} away from $\theta_0$, then their difference is supported at $\theta_0$ in $H^{-\frac{d}{2}-0}$. By \cite[Theorem 2.3.4.]{H85part0}, this difference can only be $c \delta_{\theta_0}$, as desired.
\end{proof}
\proofpart{Step}{$\mathcal{V}_{\hbar,\lambda}[\psi_{\hbar,\lambda}(\bullet;\theta_0)]=e_{\lambda}(\bullet;\theta_0)$
  for any $\theta_0 \in \mathbb{S}^{d-1}$}
Given Lemma \ref{lem:jared}, step 1 is proved by applying the
rest of the maps in the definition of $\mathcal{V}_{\hbar,\lambda}$. Recall from \eqref{eq:mainmap} and \eqref{eq:otherops}
$$\mathcal{V}_{\hbar,\lambda} \coloneqq \mathcal{I} \circ M
\circ R_{\{|\bullet|>1\}}  \circ
\widehat{\mathcal{D}}_{\frac{1}{\hbar\lambda}} \circ
\mathcal{F}_{\hbar} $$
where $R_{\{|\bullet|>1\}}(f) \coloneqq f|_{\{|\bullet|>1\}},
M(f) \coloneqq |\frac{|\bullet|^2-1}{2}|^{\frac{d+1}{2}}f,
\mathcal{I}(f) \coloneqq f(\frac{\bullet}{|\bullet|^2})$. Lemma
\ref{lem:jared} shows
$(\widehat{\mathcal{D}}_{\frac{1}{\hbar\lambda}}\circ
\mathcal{F}_{\hbar})(\psi_{\hbar,\lambda}(\bullet;\theta_0))$ is
smooth on the exterior of the unit ball, so applying the
remaining three maps is routine, and in the end one obtains
$e_{\lambda}(\bullet; \theta_0)$, as desired.
\proofpart{Step}{$\displaystyle \mathcal{V}_{\hbar,\lambda}\Big[\int_{\mathbb{S}^{d-1}}\psi_{\hbar,\lambda}(\bullet; \theta)f(\theta)d\theta \Big]=\int_{\mathbb{S}^{d-1}}\mathcal{V}_{\hbar,\lambda}[\psi_{\hbar,\lambda}(\bullet; \theta)]f(\theta)d\theta$ for all $f \in L^2(\mathbb{S}^{d-1})$, and $\mathscr{H}_{\widehat{H}_{\hbar}}(\hbar,\lambda)$ are well-defined Hilbert spaces.}
Let $f \in L^2(\mathbb{S}^{d-1})$ and define
\begin{equation*}
F_{\widehat{H}_{\hbar}}^{\hbar,\lambda}(x)\coloneqq \int_{\mathbb{S}^{d-1}}\psi_{\hbar,\lambda}(x; \theta)f(\theta)d\theta.
\end{equation*}
With \eqref{eq:linfinity} and the
preceding properties of $\psi_{\hbar,\lambda}$, we see $F_{\widehat{H}_{\hbar}}^{\hbar,\lambda}
\in C^{\infty}(\mathbb{R}^d \setminus 0) \cap C(\mathbb{R}^d)$
with $$\lVert F_{\widehat{H}_{\hbar}}^{\hbar,\lambda}\rVert_{L^{\infty}}\leq|\mathbb{S}^{d-1}| \lVert f\rVert_{L^2}^2 \lVert \psi_{\hbar,\lambda}(\bullet;\bullet)\rVert_{L^{\infty}(\mathbb{R}^d \times \mathbb{S}^{d-1})}^2.$$ We can
therefore commute $\widehat{\mathcal{D}}_{\frac{1}{\hbar\lambda}}\circ
\mathcal{F}_{\hbar}$ in $x$ with integration in $\theta$ by the
Fubini-Tonelli theorem (after pairing with $\varphi \in
\mathcal{S}(\mathbb{R}^d)$ in $x$). We have then shown
\begin{equation}\label{eq:clearly}
\Big((\widehat{\mathcal{D}}_{\frac{1}{\hbar\lambda}}\circ
\mathcal{F}_{\hbar})[F_{\widehat{H}_{\hbar}}^{\hbar,\lambda}],\varphi \Big)=\int_{\mathbb{S}^{d-1}}\Big((\widehat{\mathcal{D}}_{\frac{1}{\hbar\lambda}}\circ
\mathcal{F}_{\hbar})[\psi_{\hbar,\lambda}(\bullet;\theta)],\varphi \Big)d\theta,\quad \text{for all }\varphi \in \mathcal{S}(\mathbb{R}^d). 
\end{equation}
It remains to show that the
remaining maps $ \mathcal{I} \circ M
\circ R_{\{|\bullet|>1\}}$ commute with integration in
$\theta$. Indeed, \eqref{eq:clearly} shows
$R_{\{|\bullet|>1\}}:\mathcal{S}' \to \mathcal{S}'$  commutes
with the integration in $\theta$, by definition. The integrand
is then a smooth function in $x$ by Lemma \ref{lem:jared}, and
we conclude $\mathcal{I}$ and $M$ commute with the integration,
as desired.
\par This proof shows that the norm of
$\mathscr{H}_{\widehat{H}_{\hbar}}(\hbar,\lambda)$ is positive
definite (defined at \eqref{eq:Hilbert}) and hence defines a
Hilbert space. Indeed, if
$F_{\widehat{H}_{\hbar}}^{\hbar,\lambda} \equiv 0$, we apply
$\mathcal{V}_{\hbar,\lambda}$ on both sides and use the positive
definiteness of the norm of the Hilbert space
$\mathscr{H}_{\mathbb{H}^d}(\lambda)$ (also defined at
\eqref{eq:Hilbert}), which was proved by Helgason (see \S
\ref{sec:hyp} for the references).
\par Combining steps 1 and 2, we have shown that the diagram in Theorem \ref{th:main} commutes. Finally we address the inverse map of $\mathcal{V}_{\hbar,\lambda}$.
\proofpart{Step}{$\mathcal{V}_{\hbar,\lambda}$ is unitary.}
Define the intermediary Hilbert space of smooth functions on
the exterior of the ball
\begin{equation}\label{eq:ihilbertrest}
\mathscr{H}_{\lambda} \coloneqq \Big\{ F_{\lambda
                                                   }
  \coloneqq
                                                   \int_{\mathbb{S}^{d-1}}\frac{2^{\frac{d+1}{2}}}{(|\bullet|^2-1)^{1+\lambda i}|\bullet-\theta_0|^{d-1- 2\lambda i}}f(\theta)d\theta \mid f
                                                   \in
                                                   L^2(\mathbb{S}^{d-1})
                                                   \Big\},\ \lVert F_{\lambda}\rVert_{\mathscr{H}_{\lambda}}\coloneqq
\lVert f \rVert_{L^2},
\end{equation}
and the intermediary Hilbert space of tempered distributions on $\mathbb{R}^d$
\begin{equation}\label{eq:ihilbert}
\widetilde{\mathscr{H}}_{\lambda} \coloneqq
                                                   \Big\{\widetilde{F}_{\lambda} \coloneqq 
                                                   \int_{\mathbb{S}^{d-1}}(\widehat{\mathcal{D}}_{\frac{1}{\hbar\lambda}}\circ
\mathcal{F}_{\hbar})(\psi_{\hbar,\lambda}(\bullet;\theta))f(\theta)d\theta \mid f
                                                   \in
                                                   L^2(\mathbb{S}^{d-1})
                                                   \Big\},\ \lVert \widetilde{F}_{\lambda}\rVert_{\widetilde{\mathscr{H}}_{\lambda}}\coloneqq \lVert f \rVert_{L^2}.
\end{equation}
The argument at the end of step 2 shows these are indeed Hilbert
spaces. Step 2 also shows that
$\mathcal{V}_{\hbar,\lambda}:\mathscr{H}_{\widehat{H}_{\hbar}}(\hbar,\lambda)
\to \mathscr{H}_{\mathbb{H}^d}(\lambda)$ can be decomposed as
$$\mathscr{H}_{\widehat{H}_{\hbar}}(\hbar,\lambda)
\xrightarrow{\widehat{\mathcal{D}
  }_{\frac{1}{\hbar\lambda}}\circ \mathcal{F}_{\hbar}}
\widetilde{\mathscr{H}}_{\lambda}
\xrightarrow{R_{\{|\bullet|>1\}}} \mathscr{H}_{\lambda}
\xrightarrow{\mathcal{I} \circ M}
\mathscr{H}_{\mathbb{H}^d}(\lambda).$$
Each map preserves the norm, so $\mathcal{V}_{\hbar,\lambda}$ is
an isometry. To show it is surjective (and hence unitary), we
produce the inverse map. Define
$$\mathcal{V}_{\hbar,\lambda}^{-1}:\mathscr{H}_{\mathbb{H}^d}(\lambda)
\to \mathscr{H}_{\widehat{H}_{\hbar}}(\hbar,\lambda),\quad
\mathcal{V}_{\hbar,\lambda}^{-1} \coloneqq
\mathcal{F}_{\hbar}^{-1} \circ \widehat{\mathcal{D} }_{\hbar \lambda }\circ E_{\lambda} \circ
M^{-1}\circ \mathcal{I}$$
where the extension operator $E_{\lambda }$ is defined by
$$E_{\lambda}:\mathscr{H}_{\lambda} \to
\widetilde{\mathscr{H}}_{\lambda},\quad
E_{\lambda}(F_{\lambda})\coloneqq \widetilde{F}_{\lambda},$$
in the notation of \eqref{eq:ihilbertrest} and \eqref{eq:ihilbert}. That is, we extend the functions originally defined on the exterior of the unit ball to distributions on all of $\mathbb{R}^d$ according to the distribution in Lemma \ref{lem:jared}. One can show that $\mathcal{V}_{\hbar,\lambda}^{-1}$ is indeed the inverse of $\mathcal{V}_{\hbar,\lambda}$, and we are done.
\section{Appendix: Verifying Theorem \ref{th:main} for Spherical
  Harmonics}
The following proposition gives the outputs of the Poisson operators applied to the spherical harmonics referenced in \S\ref{sec:back}.
\begin{prop}\label{prop:main}
With the Poisson operators defined in \eqref{eq:poissonops}, we have
\begin{align}
P_{\widehat{H}_{\hbar}}(\hbar,\lambda)[Y_{\ell}^{\boldsymbol{m}}](x)&=\psi_{\hbar,\lambda,\ell}^{\boldsymbol{m}}(x) \coloneqq \frac{c_{\hbar,\lambda}2^{d-1}\pi^{\frac{d-1}{2}}\Gamma(\frac{d-1}{2}+\ell-\lambda
                               i)}{\Gamma(\frac{d-1}{2}-\lambda
                                                                      i)} \nonumber
  \\
  &\qquad \qquad \Big(\frac{2i|x|}{\hbar^2\lambda}\Big)^{\ell}e^{i\frac{|x|}{\hbar^2\lambda}}{\bf M}\Big(\frac{d-1}{2}+\ell-\lambda i;\ d -1 +2\ell;\ -\frac{2i|x|}{\hbar^2\lambda}\Big)Y_{\ell}^{\boldsymbol{m}}(\widehat{x})\label{eq:legend2}\\
  P_{\mathbb{H}^d}(\lambda)[Y_{\ell}^{\boldsymbol{m}}](\xi) &=(2\pi)^{\frac{d}{2}}\frac{\Gamma(\frac{d-1}{2}-i\lambda+\ell)}{\Gamma(\frac{d-1}{2}-i\lambda)}\bigg(\frac{2|\xi|}{1-|\xi|^2} \bigg)^{1-\frac{d}{2}}P_{-\frac{1}{2}+i \lambda }^{-(\frac{d}{2}-1+\ell)}\Big(\frac{1+|\xi|^2}{1-|\xi|^2}\Big)Y_{\ell}^{\boldsymbol{m}}(\widehat{\xi}),\label{eq:legend}
\end{align}
where $P_{\bullet}^{\bullet}$ denotes Legendre's associated $P$ functions, and $c_{\hbar,\lambda} \coloneqq |\lambda|^{\frac{d}{2}-1}\hbar^{\frac{d}{2}}\sqrt{2\pi}$.
\end{prop}
First, we recall several facts about the Gegenbauer polynomials
$C_{\ell}^{\frac{d-2}{2}}$. They are defined by
\begin{equation*}
C_{\ell}^{\frac{d-2}{2}}(t) \coloneqq \sum_{k=0}^{\lfloor \ell/2 \rfloor}(-1)^k \frac{(\frac{d-2}{2})_{\ell-k}}{k!(\ell-2k)!}(2t)^{\ell-2k},\quad 
\end{equation*}
where $(a)_k \coloneqq a(a+1) \cdots (a+k-1)=\frac{\Gamma(a+k)}{\Gamma(a)}$, and these polynomials satisfy
\begin{equation*}
\sum_{\boldsymbol{m}\in \Sigma_{\ell}}Y_{\ell}^{\boldsymbol{m}}(\theta)\overline{Y_{\ell}^{\boldsymbol{m}}(\theta')}=\frac{\frac{d}{2}-1+\ell}{(\frac{d}{2}-1)|\mathbb{S}^{d-1}|}C_{\ell}^{\frac{d-2}{2}}(\theta \cdot \theta'),\qquad t^n=\frac{n!}{2^n}\sum_{k=0}^{\lfloor n/2\rfloor}\frac{\frac{d-2}{2}+n-2k}{k!(\frac{d-2}{2})_{n-k+1}}C_{n-2k}^{\frac{d-2}{2}}(t)
\end{equation*}
where
$|\mathbb{S}^{d-1}|=\int_{\mathbb{S}^{d-1}}d\theta=\frac{2\pi^{\frac{d}{2}}}{\Gamma(\frac{d}{2})}$ is the surface area of the sphere
and $\lVert Y_{\ell}^{\boldsymbol{m}} \rVert_{L^2}=1$ (see
\cite[(4.34)-(4.35)]{T11part2}). We have the following
lemma.
\begin{lem}\label{lem:little}
We have
\begin{align}
\int_{\mathbb{S}^{d-1}}Y_{\ell}^{\boldsymbol{m}}(\theta)e^{iy\cdot
  \theta}d\theta&=(2\pi)^{\frac{d}{2}}|y|^{1-\frac{d}{2}}i^{\ell}J_{\frac{d}{2}-1+\ell}\big(|y|\big)Y_{\ell}^{\boldsymbol{m}}(\widehat{y})\\
  \int_{\mathbb{S}^{d-1}}Y_{\ell}^{\boldsymbol{m}}(\theta)\bigg(\frac{1-|\xi|^2}{|\xi-\theta|^2}\bigg)^{\alpha}d\theta &=(2\pi)^{\frac{d}{2}}\frac{\Gamma(\alpha+\ell)}{\Gamma(\alpha)}\bigg(\frac{2|\xi|}{1-|\xi|^2} \bigg)^{1-\frac{d}{2}}P_{\frac{d}{2}-1-\alpha}^{-(\frac{d}{2}-1+\ell)}(\frac{1+|\xi|^2}{1-|\xi|^2})Y_{\ell}^{\boldsymbol{m}}(\widehat{\xi}),\label{eq:legend}
\end{align}
for $\alpha \in \mathbb{C} \setminus \mathbb{Z}_{\leq 0}$ and $|\xi|<1$, where $J_{\bullet}$ are the Bessel functions of the first kind.
\end{lem}
\begin{rem}
This lemma and the preceding proposition can be easily verified
numerically with the Funk-Hecke formula:
\begin{equation}\label{eq:FunkHecke}
\int_{\mathbb{S}^{d-1}}f(\theta' \cdot \theta) Y_{\ell}^{\boldsymbol{m}}(\theta)d\theta=\frac{|\mathbb{S}^{d-2}|}{C_{\ell}^{\frac{d-2}{2}}(1)}\int_{-1}^1f(t)C_{\ell}^{\frac{d-2}{2}}(t)(1-t^2)^{\frac{d-3}{2}}dt \cdot Y_{\ell}^{\boldsymbol{m}}(\theta').
\end{equation}
The standard way to prove this is to first show the formula for the zonal spherical harmonics, and then use the fact that the zonal functions are a reproducing kernel for the rest of the spherical harmonics (see \cite{M66} and \cite[(7.4.153)]{T18b}).
\end{rem}
\begin{proof}
Note that
\begin{align*}
e^{iy\cdot\theta}=\sum_{n=0}^{\infty}\frac{(i|y|)^n}{n!}(\widehat{y}\cdot \theta)^n=\sum_{n=0}^{\infty}\frac{(i|y|)^n}{2^n}\sum_{k=0}^{\lfloor n/2\rfloor}\frac{\frac{d-2}{2}+n-2k}{k!(\frac{d-2}{2})_{n-k+1}}C_{n-2k}^{\frac{d-2}{2}}(\widehat{y}\cdot \theta).
\end{align*}
Now integrating this against
$Y_{\ell}^{\boldsymbol{m}}(\theta)$, we see that the only terms in the sum that remain are when $n=\ell+2j,k=j$ for $j=0,1,2,\ldots$, and we have
\begin{align*}
\int_{\mathbb{S}^{d-1}}Y_{\ell}^{\boldsymbol{m}}(\theta)e^{iy\cdot\theta}d\theta&=\Big(\frac{d}{2}-1\Big)|\mathbb{S}^{d-1}|Y_{\ell}^{\boldsymbol{m}}(\widehat{y})\sum_{n\in
                                                                                  \ell+2
                                                      \mathbb{Z}_{\geq
                                                                                  0}
                                                                                  }
                                                                                  \frac{(i|y|)^n}{2^n}\frac{1}{(\frac{n-\ell}{2})!(\frac{d-2}{2})_{\frac{n+\ell}{2}+1}}\\
&=\Big(\frac{d}{2}-1\Big)|\mathbb{S}^{d-1}|i^{\ell}\Gamma(\frac{d-2}{2})Y_{\ell}^{\boldsymbol{m}}(\widehat{y})\sum_{j=0
  }^{\infty}
  (-1)^j\bigg(\frac{|y|}{2}\bigg)^{\ell+2j}\frac{1}{j!\Gamma(\frac{d}{2}+\ell+j)}\\
  &=(2\pi)^{d/2}|y|^{1-\frac{d}{2}}i^{\ell}J_{\frac{d}{2}-1+\ell}\big(|y|\big)Y_{\ell}^{\boldsymbol{m}}(\widehat{y}).
\end{align*}
Similarly, note that
\begin{align*}
\bigg(\frac{1-|\xi|^2}{|\xi-\theta|^2}\bigg)^{\alpha}&=\bigg(\frac{1-|\xi|^2}{1+|\xi|^2}
                                                       \bigg)^{\alpha}\frac{1}{(1-\frac{2\xi
                                                       \cdot
             \theta}{1+|\xi|^2})^{\alpha}}=\bigg(\frac{1-|\xi|^2}{1+|\xi|^2}
                                                       \bigg)^{\alpha}\sum_{n=0}^{\infty}\binom{-\alpha}{n}\bigg(\frac{-2|\xi|}{1+|\xi|^2}\bigg)^n(\widehat{\xi}\cdot \theta)^n\\
        &=\bigg(\frac{1-|\xi|^2}{1+|\xi|^2}
                                                       \bigg)^{\alpha}\sum_{n=0}^{\infty}\binom{-\alpha}{n}\bigg(\frac{-2|\xi|}{1+|\xi|^2}\bigg)^n\frac{n!}{2^n}\sum_{k=0}^{\lfloor n/2\rfloor}\frac{\frac{d-2}{2}+n-2k}{k!(\frac{d-2}{2})_{n-k+1}}C_{n-2k}^{\frac{d-2}{2}}(\widehat{\xi}\cdot \theta),
\end{align*}
and then
\begin{align}
&\int_{\mathbb{S}^{d-1}}Y_{\ell}^{\boldsymbol{m}}(\theta)\bigg(\frac{1-|\xi|^2}{|\xi-\theta|^2}\bigg)^{\alpha}d\theta=\Big(\frac{d}{2}-1\Big)|\mathbb{S}^{d-1}|Y_{\ell}^{\boldsymbol{m}}(\widehat{\xi})\bigg(\frac{1-|\xi|^2}{1+|\xi|^2}
                                                                                                                      \bigg)^{\alpha}\Gamma(1-\alpha)\nonumber\\
                                                                                                                    &\qquad \qquad \qquad \qquad \qquad \qquad \qquad \qquad \sum_{n\in
                                                                                                                      \ell
                                                                              +
                                                                                                                      2
                                                                                                                      \mathbb{Z}_{\geq
                                                                                                                      0}}\frac{1}{\Gamma(1-n-\alpha)}\frac{1}{(\frac{n-\ell}{2})!(\frac{d-2}{2})_{\frac{n+\ell}{2}+1}}\bigg(\frac{-|\xi|}{1+|\xi|^2}\bigg)^n\nonumber\\
  &=2\pi^{d/2}Y_{\ell}^{\boldsymbol{m}}(\widehat{\xi})\bigg(\frac{1-|\xi|^2}{1+|\xi|^2}
                                                                                                                      \bigg)^{\alpha}\Gamma(1-\alpha)\sum_{j=0}^{\infty}\frac{1}{j!\Gamma(1-\ell-2j-\alpha)\Gamma(\frac{d}{2}+\ell+j)}\bigg(\frac{-|\xi|}{1+|\xi|^2}\bigg)^{2j+\ell}\label{eq:almostthere}
\end{align}
Now we use Euler's reflection formula and Legendre's duplication
formula to obtain
\begin{align*}
  \frac{1}{\Gamma(1-\ell-\alpha-2j)}=\frac{2^{2j}}{\Gamma(1-\ell-\alpha)}\frac{\Gamma(j+\frac{\ell+\alpha}{2})\Gamma(j+\frac{\ell+\alpha+1}{2})}{\Gamma(\frac{\ell+\alpha}{2})\Gamma(\frac{\ell+\alpha+1}{2})}.
\end{align*}
Using this identity and Whipple's formula (see \cite[Ex. 12.1-12.2]{O74} or \cite[\href{https://dlmf.nist.gov/14.3E18}{(14.3.18)}]{NIST:DLMF}), we see
\eqref{eq:almostthere} is equal to \eqref{eq:legend}.
\end{proof}
Now we are ready to prove the main proposition.
\begin{proof}[Proof of Proposition \ref{prop:main}]
The second statement follows immediately from Lemma
\ref{lem:little} with $\alpha=\frac{d-1}{2}-\lambda i$. For the
first statement, we again use Lemma \ref{lem:little} to obtain
\begin{align*}
P_{\widehat{H}_{\hbar}}(\hbar,\lambda)[Y_{\ell}^{\boldsymbol{m}}](x)&=
  \lim_{\varepsilon
  \to 0^+}\frac{c_{\hbar,\lambda}(2\pi)^{\frac{d}{2}}|y|^{1-\frac{d}{2}}i^{\ell}Y_{\ell}^{\boldsymbol{m}}(\widehat{y})e^{i|y|}}{\Gamma(\varepsilon+\lambda
                                                                      i)\Gamma(\frac{d-1}{2}-\varepsilon-\lambda i)}\\
  &\qquad \qquad \qquad \qquad \qquad \int_0^1e^{-i t|y|}J_{\frac{d}{2}-1+\ell}\big(|y|t\big)(1-t)^{\varepsilon+\lambda
    i-1}t^{-\frac{1}{2}-\varepsilon-\lambda i}dt
\end{align*}
where $y=\frac{x}{\hbar^2\lambda}$ and $\lambda>0$. This integral is well-known in special function theory (see \cite[(6.625) 1.]{GR07}. We have that $P_{\widehat{H}_{\hbar}}(\hbar,\lambda)[Y_{\ell}^{\boldsymbol{m}}](x)$ is equal to 
\begin{align*}
&
  \lim_{\varepsilon
  \to 0^+}\frac{c_{\hbar,\lambda}(2\pi)^{\frac{d}{2}}|y|^{1-\frac{d}{2}}i^{\ell}Y_{\ell}^{\boldsymbol{m}}(\widehat{y})e^{i|y|}}{\Gamma(\varepsilon+\lambda
  i)\Gamma(\frac{d-1}{2}-\varepsilon-\lambda
                                                                      i)}\frac{2^{-\frac{d}{2}+1-\ell}|y|^{\frac{d}{2}-1+\ell}\Gamma(\frac{d-1}{2}+\ell-\lambda i)\Gamma(\varepsilon+\lambda i)}{\Gamma(\frac{d-1}{2}+\ell+\varepsilon)\Gamma(\frac{d}{2}+\ell)}\\
  &\qquad \qquad \qquad{}_2F_2\Big(\frac{d-1}{2}+\ell-\lambda i,
    \frac{d-1}{2}+\ell;\ \frac{d-1}{2}-\ell+\varepsilon,d-1
    +2\ell; -2i|y|\Big)\\
  &=\frac{c_{\hbar,\lambda}2^{d-1}\pi^{\frac{d-1}{2}}\Gamma(\frac{d-1}{2}+\ell-\lambda i)}{\Gamma(\frac{d-1}{2}-\lambda i)} (2i|y|)^{\ell}e^{i|y|}{\bf M}\Big(\frac{d-1}{2}+\ell-\lambda i;\ d -1 +2\ell; -2i|y|\Big)Y_{\ell}^{\boldsymbol{m}}(\widehat{y}),
\end{align*}
as desired.
\end{proof}
Finally, we verify $\mathcal{V}_{\hbar,\lambda}$ takes \eqref{eq:legend2} to \eqref{eq:legend}, completing the diagram of Theorem \ref{th:main} for $Y_{\ell}^{\boldsymbol{m}}$. Suppose without loss of generality that $\lambda>0$, and define
\begin{align*}
g_{\ell}(x) \coloneqq (i|x|)^{\ell}e^{i|x|}{\bf M}\Big(\frac{d-1}{2}+\ell-\lambda i;\ d -1 +2\ell;\ -2i|x|\Big).
\end{align*}
For $|\xi|>1$, we have
\begin{align*}
  &\widehat{g_{\ell}}(\xi)=|\xi|^{1-\frac{d}{2}}(-i)^{\ell}\int_0^{\infty}g_{\ell}(r)r^{\frac{d}{2}}J_{\frac{d}{2}-1+\ell}(|\xi|r)dr\\
  &=\frac{|\xi|^{1-\frac{d}{2}}}{|\Gamma(\frac{d-1}{2}+\ell-\lambda
    i)|^2}\int_0^{1}t^{\frac{d-1}{2}+\ell-\lambda
    i-1}(1-t)^{\frac{d-1}{2}+\ell+\lambda
    i-1}\int_0^{\infty}e^{(1-2t)ir}r^{\frac{d}{2}+\ell}J_{\frac{d}{2}-1+\ell}(|\xi|r)
    drdt  \\
  &=-\frac{2^{\frac{d}{2}+\ell}|\xi|^{\ell}\Gamma(\frac{d+1}{2}+\ell)i}{\sqrt{\pi}|\Gamma(\frac{d-1}{2}+\ell-\lambda
    i)|^2}\underbrace{\int_0^{1}\frac{t^{\frac{d-1}{2}+\ell-\lambda
    i-1}(1-t)^{\frac{d-1}{2}+\ell+\lambda
    i-1}(1-2t)}{(|\xi|^2-(1-2t)^2)^{\frac{d+1}{2}+\ell}}dt}_{\eqqcolon I_{\lambda, \ell}(\xi)}
\end{align*}
where we use \cite[(6.623) 2.]{GR07} for the third equality (of course, strictly speaking, one cannot apply this identity unless the real part of the exponent is strictly negative, but one can make this rigorous via the standard way of introducing decay). We split this integral at $t=1/2$ and do the substitution $u=\ln(t^{-1}-1)$ for $t \in [0, 1/2]$ and $u=-\ln(t^{-1}-1)$ for $t \in [1/2,1]$. We have $t(1-t)=\frac{1}{2(\cosh u -1)},(1-2t)dt= -\frac{\sinh u }{2(\cosh u -1)^2}du$ in both integrals, so we have
\begin{align*}
 I_{\lambda,\ell}(\xi)&=\frac{i}{2^{\frac{d-1}{2}+\ell-1}}\int_0^{\infty}\frac{\big(\frac{1}{\cosh
                     u-1}\big)^{\frac{d-1}{2}+\ell-1}}{(\frac{2}{\cosh
                     u-1}+|\xi|^2-1)^{\frac{d+1}{2}+\ell}}\sin(\lambda u)\frac{\sinh u}{(\cosh u-1)^2}du\\
  &=\frac{i}{2^{\frac{d-1}{2}+\ell-1}(|\xi|^2-1)^{\frac{d+1}{2}+\ell}}\int_0^{\infty}\frac{\sin(\lambda
    u)\sinh u }{(\frac{|\xi|^2+1}{|\xi|^2-1}+\cosh u
    )^{\frac{d+1}{2}+\ell}}du\\
  &\overset{\mathclap{\text{I.B.P.}}}{=}\frac{\lambda i }{2^{\frac{d-1}{2}+\ell-1}(|\xi|^2-1)^{\frac{d+1}{2}+\ell}(\frac{d-1}{2}+\ell)}\int_0^{\infty}\frac{\cos(\lambda u)}{(\frac{|\xi|^2+1}{|\xi|^2-1}+\cosh u )^{\frac{d-1}{2}+\ell}}du
\end{align*}
Now we use the formula (see \cite[\href{https://dlmf.nist.gov/14.12E4}{(14.12.4)}]{NIST:DLMF})
\begin{align*}
P_{-\frac{1}{2}+\lambda i}^{-\mu}(x)=(x^2-1)^{\frac{1}{2}\mu}\frac{\sqrt{2}\Gamma(\mu+\frac{1}{2})}{\sqrt{\pi}|\Gamma(\mu+\frac{1}{2}+\lambda i)|^2}\int_0^{\infty}\frac{\cos \lambda t}{(x+\cosh t)^{\mu +\frac{1}{2}}}dt
\end{align*}
to obtain
\begin{align*}
2^{\ell}\widehat{g_{\ell}}(\xi)=\frac{2\lambda (2|\xi|)^{1-\frac{d}{2}}}{(|\xi|^2-1)^{\frac{3}{2}}}P_{-\frac{1}{2}+\lambda i}^{-(\frac{d}{2}-1+\ell)}(\frac{|\xi|^2+1}{|\xi|^2-1}).
\end{align*}
Once we see
\begin{align*}
(D_{\frac{1}{\hbar \lambda}}\circ
  \mathcal{F}_{\hbar})(\psi_{\hbar,\lambda,\ell,\boldsymbol{m}})(\xi)&=\frac{2^{d-\frac{1}{2}}\pi^{\frac{d}{2}}}{\lambda}\frac{\Gamma(\frac{d-1}{2}+\ell-\lambda
                                                                  i)}{\Gamma(\frac{d-1}{2}-i\lambda)}
                                                                  2^{\ell}\widehat{g_{\ell}}(\xi)Y_{\ell}^{\boldsymbol{m}}( \widehat{\xi})\\
                                                                &=\frac{2^{d+\frac{1}{2}}\pi^{\frac{d}{2}}\Gamma(\frac{d-1}{2}+\ell-\lambda i)}{\Gamma(\frac{d-1}{2}-i\lambda)} \frac{(2|\xi|)^{1-\frac{d}{2}}}{(|\xi|^2-1)^{\frac{3}{2}}}P_{-\frac{1}{2}+\lambda i}^{-(\frac{d}{2}-1+\ell)}(\frac{|\xi|^2+1}{|\xi|^2-1})Y_{\ell}^{\boldsymbol{m}}( \widehat{\xi}),
\end{align*}
verifying the result from the rest of the maps in $\mathcal{V}_{\hbar,\lambda}$ is routine.
\bibliographystyle{alpha-reverse}
\bibliography{refs}

\end{document}